\documentclass[10pt]{amsart}
\usepackage[all]{xy}
\usepackage{amsmath}
\usepackage{amsfonts}
\usepackage{amssymb}
\usepackage{amscd}
\usepackage{amsthm}
\usepackage{latexsym}
\usepackage{amsbsy}
\usepackage{graphicx}
\usepackage{epstopdf}
\usepackage{seqsplit}
\usepackage{color}
\usepackage{float}
\AppendGraphicsExtensions{.gif}
\newtheorem{thm}{Theorem}[section]
\newtheorem{cor}[thm]{Corollary}
\newtheorem{lem}[thm]{Lemma}
\newtheorem{prop}[thm]{Proposition}

\theoremstyle{remark}

\theoremstyle{definition}

\newcommand{\C}{\mathbb{C}}

\newcommand{\R}{\mathbb{R}}

\newcommand{\Z}{\mathbb{Z}}

\newcommand{\T}{\mathbb{T}}

\newcommand{\NT}{\mathbb{T}_\theta^2}

\newcommand{\CNT}{C(\mathbb{T}_\theta^2)}

\newcommand{\SNT}{C^\infty(\mathbb{T}_\theta^2)}

\newcommand{\Tr}{{\textrm{Tr}}}

\newcommand{\done}{\delta_1}

\newcommand{\dtwo}{\delta_2}

\newcommand{\del} {\delta}

\newcommand{\vphi}{\varphi}

\newcommand{\pv}{\partial_{\varphi}}

\newcommand{\De}{D_{e, \sigma}^+}

\newcommand{\Hc}{\mathcal{H}}

\title{A twisted local index formula for curved noncommutative two tori}
\author{Farzad Fathizadeh,$^1$  Franz Luef,$^2$ Jim Tao$\,^3$}
\date{}

\begin{document}
\begin{abstract}
We consider the Dirac operator of a general metric in the 
canonical conformal class on the noncommutative two torus, 
twisted by an idempotent (representing the $K$-theory class 
of a general noncommutative vector bundle), and derive a local 
formula for the Fredholm index of the twisted Dirac operator. Our 
approach is based on the McKean-Singer index formula, and 
explicit heat expansion calculations by making use of Connes' 
pseudodifferential calculus. As a technical tool, a new rearrangement 
lemma is proved to handle challenges posed by the noncommutativity of 
the algebra and the presence of an idempotent in the calculations in addition 
to a conformal factor. 
\end{abstract}

\maketitle

\footnotetext[1]{Department of Mathematics, Computational Foundry, 
Swansea University Bay Campus, SA1 8EN, Swansea, United Kingdom; 
Max Planck Institute for Biological Cybernetics, 72076 T\"ubingen, Germany, 
E-mail: farzad.fathizadeh@swansea.ac.uk}
\footnotetext[2]{Department of Mathematical Sciences, Norwegian University of Science and Technology, 
7491 Trondheim, Norway, E-mail: franz.luef@ntnu.no}
\footnotetext[3]{Department of Mathematics, California Institute of Technology (Caltech), 
MC 253-37, Pasadena, CA 91125, USA, E-mail: jtao@caltech.edu}

\tableofcontents

\section{Introduction}

The celebrated Atiyah-Singer index theorem provides a local formula that calculates 
the  analytical (Fredholm) index of the Dirac operator on a 
compact spin manifold, with coefficients in a vector bundle,  in terms 
of topological invariants, namely characteristic classes \cite{MR0157392, MR0236950, MR0236951, MR0236952}. The theorem has 
deep applications in mathematics and theoretical physics. For instance,  
it implies important theorems such as the Gauss-Bonnet theorem and the 
Riemann-Roch theorem, and it is used in physics to count linearly the number of 
solutions of  fundamental partial differential equations. 

In noncommutative differential geometry \cite{MR823176, MR1303779}, an analog of the Dirac operator is used to 
encode the metric information and to use ideas from spectral geometry and Riemannian 
geometry in studying a noncommutative algebra, viewed as the algebra of functions on 
a space with noncommuting coordinates. 
That is, the data $(C^\infty(M), L^2(S), D)$ of smooth functions on a spin manifold, 
the Hilbert space of $L^2$-spinors, and the Dirac operator is extended to the notion 
of a {\em spectral triple} $(A, \mathcal{H}, D)$ where $A$ is a noncommutative algebra 
acting on a Hilbert sapce $\mathcal{H}$, and $D$ plays the role of the Dirac operator while acting 
in $\mathcal{H}$. The idempotents and thereby $K$-theory 
of the algebra is also used to consider the analog of vector bundles on an ordinary manifold.

The identity of the analytical index and the topological index defined by the Connes-Chern character 
was established in the noncommutative setting in \cite{MR823176}. Moreover, a local formula for the 
Connes-Chern character, which is suitable for explicit calculations, was derived in \cite{MR1334867}. However, as 
we shall elaborate further in \S \ref{twistedspectripsec}, the notion of a spectral triple is suitable for 
type II algebras in Murray-von Neumann classification of algebras, and not for type III cases. 
The remedy brought forth in \cite{MR2427588} for studying the latter is the notion of a {\em twisted spectral triple}: 
they have shown that a twisted version of spectral triples can incorporate type III cases and examples 
that arise in noncommutative conformal  geometry, and that the index pairing and the coincidence of the 
analytic and the topological index continues to hold in the twisted case.  However, a general local formula 
for the index in the twisted case has not yet been found, except for the special case of twists afforded by scaling 
automorphism in \cite{MR2732069}. In fact, the difficulty of the problem of finding a general  local index formula for 
twisted spectral triples raises the need for treating more examples.

The main result of the present article is a local formula for the index of the Dirac operator 
$D^+_{e, \sigma}$ 
of a twisted spectral triple on the opposite algebra of the noncommutative two torus $\NT$, which is 
twisted by a general noncommutative vector bundle represented by an idempotent $e$. 
The Dirac operator that we consider is associated with a general metric in the canonical 
conformal class on $\NT$. Our approach is based on using the McKean-Singer index formula \cite{MR0217739} 
and performing explicit heat kernel calculations by employing Connes' pseudodifferential calculus 
developed in \cite{MR572645} for $C^*$-dynamical systems.  
It should be noted that, following the Gauss-Bonnet theorem proved in \cite{MR2907006} for the canonical conformal 
class on $\NT$ and its extension in \cite{MR2956317} to general conformal classes, study of local geometric invariants 
of curved metrics on noncommutative tori has received remarkable attention in recent years 
\cite{MR3194491, MR3148618, MR3540454, MR3359018, MR3402793, MR3825195, 2016arXiv161109815C}.  
For an overview of these developments one can refer to \cite{2019arXiv190107438F, 2018arXiv181010394L}. 

This article is organized as follows. In \S \ref{prelsec} we provide background material about the
noncommutative two torus, its pseudodifferential calculus, and heat kernel methods. In \S \ref{SpecTripSec}
we explain our construction of a vector bundle over the noncommutative two torus using a twisted
spectral triple with $\sigma$-connections and derive the symbols of the operators
necessary for the index calculation. In \S \ref{calculationssec} we calculate the explicit local terms that 
give the index.   
The appendix contains proofs of new rearrangement lemmas, which overcome the new challenges 
posed in our calculations due to the presence of an idempotent as well as a conformal factor in our 
calculations in the noncommutative setting.

\section{Preliminaries}
\label{prelsec}

In this section we provide some background material about the noncommutative 
two torus $\NT$, and explain how one can use a noncommutative 
pseudodifferential calculus to derive the heat kernel expansion for an elliptic operator 
on $\NT$. 

\subsection{Noncommutative two torus} 

For a fixed irrational number $\theta$, the noncommutative two torus 
$\NT$ is a {\em noncommutative manifold} whose algebra of {\em continuous} 
functions $\CNT$ is the universal (unital) $C^*$-algebra generated by two unitary 
elements $U$ and $V$ that satisfy the commutation relation 
\[
VU = e^{2 \pi i \theta} UV. 
\] 
The ordinary two torus $\T^2 = \left ( \R / 2 \pi \Z \right )^2$ acts on $\CNT$ by 
\begin{equation} \label{torusaction}
\alpha_s \left( U^m V^n \right) = e^{i s \cdot (m, n)} \, U^m V^n, \qquad s \in \T^2, \qquad m, n \in \Z. 
\end{equation}
Since this action, in principle, comes from translation by $s \in \T^2$ written in the Fourier mode, 
its infinitesimal  generators  $\done, \dtwo : \SNT \to \SNT$ are the derivations that are analogs of partial 
differentiations on the ordinary two torus and are given by the  
defining relations 
\[
\done(U) = U, \qquad \done(V) =0, \qquad \dtwo(U) =0, \qquad \dtwo(V)=V.
\]

The space $\SNT$ of the {\em smooth} elements consists of all elements in $\CNT$ that are smooth with respect 
to the action $\alpha$ given by \eqref{torusaction}, which turns out to be a dense 
subalgebra  of $\CNT$ that can alternatively be described as the space of 
all elements of the form $\sum_{m, n \in \Z} a_{m, n} U^m V^n$ with rapidly decaying 
complex coefficients $a_{m, n}$. The analog of integration is provided by the linear 
functional $\tau : \CNT \to \C$ which is the bounded extension of the linear functional 
that sends any smooth element with the {\em noncommutative Fourier expansion} 
$\sum_{m,n} a_{m, n} U^m V^n$ to its constant term $a_{0,0}$. The functional $\tau$ 
turns out to be a positive trace  on $\CNT$ and we shall view it as the volume form of a flat 
canonical metric on $\NT$. In \S \ref{NCconformalsec}, we will explain how one can consider 
a general metric in the conformal class of the flat canonical metric on $\NT$ by means of a 
positive invertible element $e^{-h}$, where $h$ is a self-adjoint element in $\SNT$.

\subsection{Heat kernel expansion}
\label{heatkerexpansiongensec}
We will explain in \S \ref{IndexThmSec} that our method of calculation of a 
noncommutative local index formula, which is the main result of this article, 
is based on the McKean-Singer index formula \cite{MR0217739} and calculation of the relevant terms 
in small time heat kernel expansions. Here we elaborate on the derivation of the 
small heat kernel expansion for an elliptic 
differential operator of order 2 on $\CNT$. The main tool that will be used is the 
pseudodifferential calculus developed in 
\cite{MR572645} for $C^*$-dynamical systems (see also \cite{2018JPhCS.965a2042T, 2018arXiv180303575H}). This 
calculus, in the case of $\NT$, associates to a {\em pseudodifferential symbol} $\rho : \R^2 \to \SNT$ 
a pseudodifferential operator $P_\rho : \SNT \to \SNT$ by the formula
\[
P_{\rho}(a) 
=
\frac{1}{(2\pi)^2} \int_{\mathbb R^2}\!\int_{\mathbb R^2}\! e^{-is\cdot\xi}\rho(\xi) \, \alpha_s(a)\, ds\, d\xi, 
\qquad  a \in \SNT. 
\]
For example, any differential operator given by a finite sum of the form 
$\sum a_{j_1, j_2} \done^{j_1} \dtwo^{j_2}$ with $a_{j_1, j_2} \in \SNT$ 
is associated with the polynomial $\sum a_{j_1, j_2} \xi_1^{j_1} \xi_2^{j_2}$. In general 
a smooth map $\rho : \R^2 \to \SNT$ is a pseudodifferential symbol of order $m \in \Z$ 
if for any non-negative integers $i_1, i_2, j_1, j_2$ there is a constant $C$ such that 
\[
|| \partial_1^{j_1} \partial_2^{j_2}\done^{i_1} \dtwo^{i_2} \rho(\xi)  || \leq C (1+|\xi|)^{m-j_1-j_2},  
\]
where $\partial_1$ and $\partial_2$ are respectively the partial differentiations with 
respect to the coordinates  $\xi_1$ and $\xi_2$ of $\xi \in \R^2$. We denote the space 
of symbols of order $m$ by $S^m$. A symbol $\rho$ of order $m$ is {\em elliptic} if 
$\rho(\xi)$ is invertible for larger enough $\xi$ and there is a constant $c$ such that 
\[
||\rho(\xi)^{-1}|| \leq c (1+|\xi|)^{-m}.
\]
In the case of differential operators, one can see that the symbol is elliptic if 
its leading part is invertible away from the origin.

An important feature of an elliptic operator is that it admits a {\em parametrix}, namely 
an inverse in the algebra of pseudodifferential operators modulo infinitely smoothing 
operators. The latter are the operators whose symbols belong to the intersection of 
all symbols $S^{-\infty} = \cap_{m \in \Z} S^m$. This illuminates the importance of 
pseudodifferential calculus for solving elliptic differential equations.

For our purposes in this article, it is crucial to illustrate that given a positive elliptic differential operator 
$\triangle$ of order 2 on $\SNT$,  one can use the pseudodifferential calculus to derive a small 
time asymptotic expansion for the trace of the heat kernel $\exp(-t \triangle)$. Note that the 
pseudodifferential symbol of $\triangle$ is of the form 
\[
\sigma_\triangle(\xi) = p_2(\xi)+p_1(\xi)+p_0(\xi),
\] 
where each $p_k$ is homogeneous of order $k$ in $\xi$. Since the eigenvalues of $\triangle$ are 
real and non-negative, using the Cauchy integral formula and a clockwise contour $\gamma$ that 
goes around the non-negative real line, one can write
\begin{equation} \label{Cauchyformula}
\exp(-t \triangle) = \frac{1}{2 \pi i} \int_\gamma e^{-t \lambda} (\triangle - \lambda)^{-1} \, d\lambda. 
\end{equation}
Since $\triangle - \lambda$ is an elliptic operator, its parametrix $R_\lambda$ 
will approximate $(\triangle - \lambda)^{-1}$ appearing above in the integrand.  
Since $\triangle - \lambda$ is of order 2, one can write for the symbol of 
$R_\lambda$: 
\[
\sigma_{R_\lambda}(\xi, \lambda, \triangle) = b_0(\xi, \lambda, \triangle) + b_1(\xi, \lambda, \triangle)  + b_2(\xi, \lambda, \triangle) + \cdots, 
\]
where each $r_j$ is of order $-2-j$. 

Then, by considering the ellipticity of $\triangle - \lambda$ and 
the following composition rule, which gives an asymptotic expansion for the symbol of the composition 
of two pseudodifferential operators, 
\[
\sigma_{P_1  P_2}  
\sim 
\sum_{i_1, i_2 \in \Z_{\geq 0}} \frac{1}{i_1! i_2!}\, \partial_1^{i_1} \partial_2^{i_2} \sigma_{P_1} \, 
\done^{i_1} \dtwo^{i_2} \sigma_{P_2},  
\]
one can find a recursive formula for the $r_j$. That is, it turns out that 
\begin{equation} \label{r_0formula}
b_0(\xi, \lambda, \triangle) = (p_2(\xi) - \lambda)^{-1}, 
\end{equation}
and for any $n \in \Z_{\geq 1}$, 
\begin{equation} \label{r_nformula}
b_n(\xi, \lambda, \triangle)  = - \left ( 
\sum_{\substack{i_1+i_2+j+2-k=n \\ 0 \leq j<n, \, 0 \leq k \leq 2}} \frac{1}{i_1! i_2!}
\partial_1^{i_1} \partial_2^{i_2} b_j (\xi, \lambda, \triangle)  \, \done^{i_1} \dtwo^{i_2} (p_k(\xi))  
\right ) 
b_0 (\xi, \lambda, \triangle). 
\end{equation}
By calculating these terms, one can replace $(\triangle-\lambda)^{-1}$ in \eqref{Cauchyformula} 
by $R_\lambda$ and find an approximation of $\exp(-t\triangle)$.

Then, calculating the trace of this approximation, one can derive a small time asymptotic expansion for 
$\Tr(\exp(-t \triangle))$.  That is, one finds that there are elements 
\begin{equation} \label{a_2nformula}
a_{2n}(\triangle) = \frac{1}{2\pi i}  \int_{\R^2} \int_\gamma e^{-\lambda} b_{2n}(\xi, \lambda, \triangle)  \, d\lambda \,d\xi \in \SNT
\end{equation}
 such that 
as $t \to 0^+$, for any $a \in \SNT$:
\begin{equation} \label{asmpexpnct}
\Tr(a \exp(-t \triangle)) \sim t^{-1} \sum_{n=0}^\infty \tau \left (a \, a_{2n} (\triangle) \right ) \, t^n. 
\end{equation}
The terms $a_{2n}(\triangle)$ are local geometric invariants associated with the operator $\triangle$ 
when it is a natural geometric operator such as a Laplacian, or the square of a twisted Dirac operator 
as in this article.

\section{Noncommutative geometric spaces and index theory}
\label{SpecTripSec}

\subsection{Spectral triples} 
Geometric spaces are defined in terms of spectral data in 
noncommutative geometry \cite{MR1303779}. 
A noncommutative geometric 
space is a {\em spectral triple} $(A, \mathcal{H}, D)$, where 
$A$ is an involutive algebra represented by bounded operators 
on a Hilbert space $\mathcal{H}$, and $D$ is the analog of the Dirac 
operator. That is, $D$ is an unbounded self-adjoint operator on its 
domain which is dense in the Hilbert space $\mathcal{H}$, and the 
spectrum of $D$ has similar properties to the spectrum of the Dirac 
operator on a compact Riemannian spin$^c$ manifold, or even more 
generally, that of an elliptic self-adjoint differential operator of order 1 
on a compact manifold.  The {\em spectral dimension} of such a triple is 
the smallest positive real number $d$ such that $|D|^{-d}$ is in the 
domain of the {\em Dixmier trace} $\textnormal{Tr}_\omega$ 
\cite{MR0196508}
(or one can say that for any $\epsilon > 0$, the operator 
$|D|^{-d-\epsilon}$  is a trace-class operator which means that 
its eigenvalues are summable). 
Moreover, an important assumption is that the commutator of $D$ and 
the action of any $a \in A$ on $\mathcal{H}$, $[D, a]:=Da-aD$, 
extends to a bounded operator on $\mathcal{H}$.  In this paradigm, the 
algebra $A$ is allowed to be noncommutative, viewed as the algebra of 
functions on a space with noncommuting coordinates, and the metric 
information is encoded in the operator $D$.

The main classical example of this setup is the triple 
$\left ( C^\infty(M), L^2(S), D_g \right )$, where $C^\infty(M)$ is 
the algebra of smooth complex-valued functions on a compact 
Riemannian manifold with a spin$^c$ structure $S$,  $L^2(S)$ 
is the Hilbert space of the $L^2$-spinors, and $D_g$ is the Dirac 
operator associated with the metric, which acts on its domain 
in $L^2(S)$ and squares to a Laplace-type operator, see for 
example \cite{MR2273508} for details about the 
Dirac operator. Since $S$ is a vector bundle 
over $M$, the algebra $C^\infty(M)$ acts naturally by bounded 
operators on the Hilbert space $L^2(S)$, namely: 
$(f\cdot s)(x) := f(x) s(x),$ $f \in C^\infty(M), s \in L^2(S), x\in M$. 
More importantly, the latter action interacts in a bounded manner with 
the Dirac operator $D_g$ in the sense that the commutator of $D_g$ with 
the action of each $f\in C^\infty(M)$ is a bounded operator as  
$[D_g, f] =D_g f - f D_g = c(df)$, where $c(df)$ denotes the Clifford 
multiplication on $S$ by the de Rham differential of $f$.   
By construction, the Dirac operator $D_g$ depends heavily 
on the Riemannian metric $g$ on $M$. It is known from classical 
facts in spectral geometry that the spectral dimension of 
the spectral triple $\left ( C^\infty(M), L^2(S), D_g \right )$ is equal 
to the dimension of the manifold $M$, and that the important local curvature 
related information can be detected in the 
small time asymptotic expansions of the form 
\begin{equation} \label{asympexp}
\textnormal{Tr} \left (f \exp(-t D_g^2) \right )  \,\,\, 
\sim_{t \to 0^+}  
\,\,\, 
t^{- \textnormal{dim} M/2}  
\sum_{j=0}^\infty t^j \int_M f(x) \, a_{2j}(x) \, dx, \qquad f \in C^\infty(M), 
\end{equation} 
where the densities $a_{2j}(x) \, dx$ are uniquely determined by 
the Riemann curvature tensor and its contractions and covariant 
derivatives, cf. \cite{MR2371808}.

The importance of the Dirac operator, for using the tools of 
Riemannian geometry in studying noncommutative algebras, 
is fully illustrated in \cite{MR3032810} by showing that the 
Dirac operator $D_g$ contains the full metric information. That is, 
it is shown that any spectral triple $(A, \mathcal{H}, D)$ whose 
algebra $A$ is commutative,  and satisfies suitable  
conditions, is equivalent to the spectral triple 
$\left ( C^\infty(M), L^2(S), D_g \right )$ of a Riemannian spin$^c$ 
manifold. In fact, following the Gauss-Bonnet theorem for the 
noncommutative two torus proved in \cite{MR2907006} and 
its extension in \cite{MR2956317}, the calculation and 
conceptual understanding of the local curvature terms in 
noncommutative geometry has attained remarkable attention 
in recent years 
\cite{MR3194491, MR3148618, MR2947960, MR3359018, MR3369894, 2016arXiv161109815C, MR3540454, MR3402793, MR3705386, MR3825195}. 
In these studies, the noncommutative 
local geometric invariants are detected in the analogs of 
the small time asymptotic expansion 
\eqref{asympexp} written for spectral triples.

\subsection{Twisted spectral triples}   
\label{twistedspectripsec}

It turns out that the notion of  a spectral triple $(A, \mathcal{H}, D)$ is suitable for 
studying algebras that possess a non-trivial trace, and a {\em twisted} notion of 
spectral triples is proposed in \cite{MR2427588} to incorporate 
algebras that do not have this property. The reason is that 
if $(A, \mathcal{H}, D)$ is a spectral triple with spectral dimension $d$, then using the 
Dixmier trace $\textnormal{Tr}_\omega$, one can define the 
linear function $\phi : A \to \mathbb{C}$  by 
$\phi(a) = \textnormal{Tr}_\omega(a |D|^{-d})$, which turns out to be a trace, 
under the minimal regularity assumption that the commutators $[|D|, a]$ are also 
bounded for any $a \in A$. 
The main reason for the trace property $ \phi(ab)=\phi(ba), a, b \in A,$ is that 
for any $a \in A$, the commutator $[|D|^{-d},a]=|D|^{-d}a - a |D|^{-d}$ belongs to the 
ideal of compact operators on which the Dixmier trace $\textnormal{Tr}_\omega$ 
vanishes. Therefore, in order to incorporate algebras coming from type III 
examples in the Murray-von Neumann classification of algebras, which do not possess a 
non-trivial trace functional, the notion of a {\em twisted spectral triple} was introduced in 
\cite{MR2427588}. The difference is that they change the definition of 
a spectral triple $(A, \mathcal{H}, D)$ by introducing a twist 
by an algebra automorphism $\sigma : A \to A$, and by requiring that instead of the ordinary 
commutators, the twisted commutators of the form $[D, a]_\sigma := Da - \sigma(a) D$, $a \in A$, 
extend to bounded operators on the Hilbert space $\mathcal{H}$. 
We note that it is natural to assume the 
mild regularity condition that the twisted commutators with $|D|$ are also bounded operators 
and to consider a grading: a bounded selfadjoint operator $\gamma$ on $\mathcal{H}$ 
that squares to identity, commutes with the action of $A$ and anti-commutes with $D$.   
The boundedness of the twisted commutators is then used to observe that for any $a \in A$, 
the operator $|D|^{-d} a - \sigma^{-d}(a)|D|^{-d}$ is in the kernel of the Dixmier trace (which is 
an ideal), hence, the linear functional $a \mapsto \textnormal{Tr}_\omega (a |D|^{-d})$, $a \in A$ is a twisted 
trace.  That is, $\textnormal{Tr}_\omega (a \,b |D|^{-d}) = \textnormal{Tr}_\omega ( b \, \sigma^{-d}(a) |D|^{-d})$, for any $a, b \in A$.

It is emphasized in \cite{MR2427588} on the important 
issue about twisted spectral triples that their Connes-Chern character lands in 
the ordinary cyclic cohomology whose pairing 
with the $K$-theory of the algebra can be realized as the index of a Fredholm operator. 
That is, let $(A, \mathcal{H}, D, \sigma \in \textnormal{Aut}(A))$ be a twisted spectral 
triple of spectral dimension $d$. Then by passing to the phase $F=D/|D|$ one 
arrives at an ordinary {\em Fredholm module} $(A, \mathcal{H}, F)$ with the 
{\em Connes-Chern character} \cite{MR0823176},  
\begin{equation}
\label{CCcharacter}
\Phi(a_0, a_1, \dots, a_d) 
=
\textnormal{Tr}\left ( \gamma F [F, a_0] [F, a_1] \cdots [F, a_d]  \right ), \qquad a_0, a_1, \dots a_d \in A. 
\end{equation}
This multilinear functional is a {\em cyclic cocycle}, 
it pairs with the $K$-theory of $A$ as the Fredholm index 
of an operator, and the result of the pairing depends only on 
the cyclic cohomology class of $\Phi$ and the $K$-theory class of any chosen 
idempotent \cite{MR0823176, MR2427588}. 
However, for the purpose of explicit calculations, one needs to have a local 
formula that is cohomologous to $\Phi$ in cyclic cohomology or equivalently 
in the $(b, B)$-bicomplex. For ordinary spectral triples (when the 
automorphism $\sigma$ is the identity), the local formula of Connes and 
Moscovici provides the desired formula in the $(b, B)$-bicomplex 
\cite{MR1334867}, see also \cite{MR2036597}. 
As a first step towards a local formula 
for twisted spectral triples, in \cite{MR2427588} they have shown that there is a 
{\em Hochschild cocycle} associated with  twisted spectral triples, which is defined by  
\begin{equation}
\label{localhoch}
\Psi(a_0, a_1, \dots, a_d) 
=
\textnormal{Tr}_\omega \left ( \gamma a_0 [D, \sigma^{-1}(a_1)]_\sigma 
[D, \sigma^{-2}(a_2)]_\sigma \cdots [D, \sigma^{-d}(a_d)]_\sigma |D|^{-d} \right ), 
\end{equation}
for $a_0, a_1, \dots a_d \in A$.

Explicit calculations are often possible with this multi-linear 
map $\Psi$, as for example one can use the trace theorem 
of \cite{MR0953826} and its extension to noncommutative tori 
\cite{MR3004814} 
to write it in terms of local formulas. However, the issue is that $\Psi$ is only 
a Hochschild coccyle and in general it is not cyclic, therefore it does not pair 
with the $K$-theory (see \cite{MR1789831, MR2838682} for the 
relation between the Hochschild cocycle and Connes-Chern 
character in Hochschild cohomology).  Finding a twisted version 
of the local index formula of \cite{MR1334867} has so far proved 
to be a challenging problem, and it has only been done for the special 
examples of scaling automorphisms in \cite{MR2732069}. 
For the treatment of the problem in twisted cyclic cohomology one 
can refer to 
\cite{2011arXiv1111.6546K, MR2888978, MR2770560, MR3275024, MR3316712}. 
We elaborate further on index theoretic 
aspects of the present paper and its relation with 
classical results and results in noncommutative geometry in 
\S \ref{IndexThmSec}.

\subsection{Noncommutative conformal geometry}
\label{NCconformalsec}

As a motivating example it is shown in  
\cite{MR2427588}  that the twisted notion of a 
spectral triple arises naturally in noncommutative conformal 
geometry. They consider the fact that on a Riemannian spin manifold, 
the Dirac operator of the conformal perturbation $g' = e^{-4h} g$ of a 
metric $g$ is unitarily equivalent to $e^h D e^h$, where $D$ is the 
Dirac operator of $g$. Thus, starting from a spectral triple $(A, \mathcal{H}, D)$ 
they use a selfadjoint element $h \in A$ to encode the conformal perturbation 
of the metric in the operator $D'=e^h D e^h$. However, when $A$ is noncommutative, 
$(A, \mathcal{H}, D')$ is not a spectral triple any more, and one needs the twist 
$\sigma(a)=e^{2h} a e^{-2h}$, $a \in A$, to have bounded twisted commutators 
$[D', a]_\sigma.$

In fact, twisted spectral triples of this nature can 
be constructed in a more intrinsic manner as in 
\cite{MR2907006} for the noncommutative two torus 
with a conformally flat metric, and as in the extension of this construction 
in  \cite{MR3458156} to ergodic $C^*$-dynamical 
systems, see also \cite{MR3500823}. Because of its 
intimate relevance to  the present work, we review here the 
construction in \cite{MR2907006}.

Viewing the canonical normalized trace  $ \tau : \CNT \to \mathbb{C} $ 
on the noncommutative two torus as the volume form of the flat metric, an 
element $h=h^* \in \SNT$ is fixed and the positive invertible element 
$e^{-h} \in \SNT$ is used as a conformal factor to perturb the flat metric 
conformally. The volume form of the curved metric is given by the linear functional  
$\varphi : \CNT \to \mathbb{C} $ defined by 
$\varphi (a) = \tau (a e^{-h}), a \in \CNT$. This linear functional 
turns out to be a KMS state with a 1-parameter group of automorphisms. 
The Dirac operator of the curved metric is constructed by using the following analogs 
of the Dolbeault operators: 
\[
\partial = \delta_1 +i \delta_2, \qquad \bar \partial = \delta_1 - i \delta_2. 
\]
For conceptual details of the notion of complex structure 
in noncommutative geometry we refer the 
reader to \cite{MR1303779}, see also \cite{MR2773332}.
Let $\mathcal{H}^+$ be the Hilbert space obtained by the GNS construction 
from $(\CNT, \varphi)$ and consider 
\[
\partial_\varphi = \partial: \mathcal{H}^+ \to \mathcal{H}^-,
\] 
where  $\mathcal{H}^{-}$, 
the analogue of $(1, 0)$-forms, is the Hilbert space completion of 
finite sums $\sum a \partial(b), a, b \in \SNT $, 
with respect to the inner product 
\[
\langle a \partial b, c \partial d \rangle = \tau (c^* a (\partial b) (\partial d)^*). 
\]
In this setting the Dirac operator is defined on the 
Hilbert space $\mathcal{H}= \mathcal{H}^+ \oplus \mathcal{H}^-$ as:  
\[ 
D=
\left(\begin{array}{c c}
0 & \partial_\varphi^* \\
\partial_\varphi & 0
\end{array}\right)
: \mathcal{H} \to \mathcal{H},
\] 
where the adjoint $\partial_\varphi^*$ of $\partial_\varphi$ is seen to
be given by 
\cite{MR2907006} 
\[
\partial_\varphi^*(a)= \bar \partial (a) k^2, \qquad \partial_\varphi^* = R_{k^2} \bar \partial, 
\] 
where $k=\exp(h/2)$, and $R_{k^2}$ denotes the right multiplication by $k^2$.
The crucial point is that the action of $\CNT$ on $\mathcal{H}$ gives rise to an 
ordinary spectral triple while the action of the opposite algebra 
$\CNT^{op}$ of $\CNT$   leads to a twisted spectral triple  
triple  \cite{MR2907006}. The main reason for this phenomena is that 
the action of $\CNT$ on $\mathcal{H}^+$ induced by left multiplication is 
a $*$-representation of the algebra, however, the action of $\CNT^{op}$ 
on $\mathcal{H}^+$ induced by right multiplication is not a $*$-representation 
and requires a modification which can be provided by an algebra automorphism. 
That is for  $a \in \CNT^{op}$ and  $\xi^+ \in \mathcal{H}^+$, one has to define  
\[
a^{op} \cdot \xi^+ = a_+ (\xi^+) = \xi^+ k^{-1} a k,   
\]
hence the appearance of the automorphism of $\CNT^{op}$ given by 
\[
\sigma(a^{op})= k^{-1} a k, \qquad a \in \CNT, 
\]
Note that for any $\xi^- \in \mathcal{H}^-$ one stays with $a^{op} \cdot \xi^- = \xi^- a$ 
since the inner product on $(1,0)$-forms is not affected 
by the conformal perturbation.  This provides the ingredients of a twisted spectral triple, namely 
for any $a \in \CNT^{op}$ the twisted commutator $D a^{op} - \sigma(a^{op}) D$ is 
a bounded operator on $\mathcal{H}  = \mathcal{H}^+ \oplus \mathcal{H}^-$ 
(while the ordinary commutators are not necessarily bounded). 

We also consider the grading $\gamma: \mathcal{H}  = \mathcal{H}^+ \oplus \mathcal{H}^- \to \mathcal{H}  = \mathcal{H}^+ \oplus \mathcal{H}^- $ given by 
\[
\gamma = \left(\begin{array}{c c}
1 & 0 \\
0 & -1
\end{array}\right),  
\]
and for convenience write the action of any $a \in \SNT^{op}$ on $\mathcal{H}= \mathcal{H}^+ \oplus \mathcal{H}^-$ as 
\[
\pi (a) = \left(\begin{array}{c c}
a_+ & 0 \\
0 & a_- \end{array}\right),
\]
where, clearly, the operators $a^\pm: \mathcal{H}^\pm \to \mathcal{H}^\pm$ are induced by
\[
a_+ (\xi^+) =  \xi^+ k^{-1} a k, \qquad a_+ = R_{k^{-1} a k} = R_{\sigma(a)}, 
\]
\[
a_- (\xi^-) =  \xi^- a, \qquad a_- = R_{a}. 
\]
We also mention that the automorphism $\sigma$ satisfies the natural condition 
\[
\sigma(a)^* = \sigma^{-1}(a^*), \qquad \sigma^{-1} \circ * = * \circ \sigma.
\]

\subsection{Dirac operator twisted by a module (vector bundle)}
\label{Lplusminussec}

Twisting the Dirac operator  with a hermitian vector bundle by means 
of a hermitian connection is a standard technique in differential geometry.  
In the noncommutative setting, this can be carried out by compressing the operator 
$F=D/|D|$ of the Fredholm module of a twisted spectral triple 
$(A, \mathcal{H}, D, \sigma  \in \rm{Aut}(A))$ by an idempotent $e$ as in \cite{MR2427588}.  However, in order to implement 
the twisting with the operator $D$, one has to consider a twisted compression \cite{MR3500845}. That is, 
the index of the operator $\sigma(e) D e $ is important to be calculated. Note that in general 
$e$ is an idempotent matrix in $M_q(A)$ for some positive integer $q$, but for our purposes 
for the noncommutative two torus, we can consider 
$e \in \SNT$, see \cite{MR0587369, MR0703981}.  Thus, we need to consider the operator 
\[
\sigma(e) D e = \left(\begin{array}{c c}
0 & \sigma(e_+) \partial_\vphi^* e_- \\
\sigma(e_-) \partial_\vphi e_+  & 0 \end{array}\right) : 
e \mathcal{H} \to \sigma(e) \mathcal{H},  
\]
and calculate a local formula for the index of the operator 
\begin{equation}
\label{twistedDiracopnct}
D_{e, \sigma}^+ := 
\sigma(e_-) \partial_\vphi e_+ : e_+ \mathcal{H}^+ 
\to
\sigma(e_-) \mathcal{H}^-.  
\end{equation}
We will shortly explain in \S \ref{IndexThmSec} that this can be done, using the McKean-Singer 
index formula, by calculating the constant terms in the small time heat kernel expansions of the 
form \eqref{asmpexpnct} for the operators $(\De)^* \De$ and $ \De (\De)^*$. We calculate the relevant terms in the 
heat expansions by employing 
the pseudodifferential calculus and the method illustrated in \S \ref{heatkerexpansiongensec}. Therefore, the first step is to 
calculate the pseudodifferential symbols of these operators. 
It will be convenient for us to use the inclusions 
\[
\iota_+: e_+ \Hc^+ \to \Hc^+, \qquad 
\iota_-: \sigma(e_-) \Hc^- \to \Hc^-, 
\]
and the fact that 
$
\iota_+^* = e_+, \sigma(e_-)^* = \iota_-. 
$

We have: 
\[
\De = \sigma(e_-) \, \pv \, \iota_+ 
: e_+ \mathcal{H}^+ \xrightarrow{\iota_+} 
\Hc^+ \xrightarrow{\pv} \Hc^-
\xrightarrow{\sigma(e_-)} \sigma(e_-)\Hc^-.  
\]
The operator 
\[
(\De)^* \De = \iota_+^* \pv^* \sigma(e_-)^* \sigma(e_-) \pv \iota_+  
= e_+ \pv^* \iota_- \sigma(e_-) \pv \iota_+: e_+ \mathcal{H}^+  \to e_+ \mathcal{H}^+ 
\]
is the restriction of the operator 
\[
L_1^+ = e_+ \pv^*  \sigma(e_-) \pv : \Hc^+ \to \Hc^+
\]
to the subspace $e_+ \Hc^+$, which means that we can write: 
\[ 
L_1^+ = 
\left(\begin{array}{c c}
(\De)^* \De & * \\
0  & 0 \end{array}\right), 
\]
which yields 
\[ 
\exp(-t L_1^+) = 
\left(\begin{array}{c c}
\exp \left (-t \left (\De \right)^* \De \right)& * \\
0  & 0 \end{array}\right), 
\]
and
\[
\Tr \left ( \exp(-t L_1^+) \right ) = 
\Tr \left ( \exp \left (-t \left (\De \right)^* \De \right) \right ).  
\]
We have a unitary map $W : \mathcal{H}_0 \to \Hc^+$ defined by 
\[
W(a) = a k, \qquad W = R_k, 
\] 
and use it to work with the anti-unitarily equivalent 
operator
\[
L^+ = JW^* L_1^+ WJ : \Hc_0 \to \Hc_0,  
\]
where the operator $J$ is induced by the involution of $\CNT$. 
We have: 
\[
L_1^+ = R_{\sigma(e)} \,R_{k^2}\, \bar \partial \, R_{\sigma(e)}\, \partial, 
\]
\[
W^* L_1^+ W = R_{k^{-1}} \, R_{\sigma(e)} \,R_{k^2}\, \bar \partial \, R_{\sigma(e)}  \, \partial \, R_{k}, 
\]
and 
\begin{eqnarray*}
L^+ &=& J W^* L_1^+ W J \\
 &=&  JR_{k^{-1}} \, R_{\sigma(e)} \,R_{k^2}\, \bar \partial \, R_{\sigma(e)} \, \partial \, R_{k} J \\
&=& (JR_{k^{-1}}J) \,(J R_{\sigma(e)}J) \, (JR_{k^2} J)\, (J\bar \partial J)\, (JR_{\sigma(e)}J) \, (J\partial J) \, (J R_{k} J) \\
&=& k^{-1} \, \sigma(e) \, k^2 \, (-\partial ) \, \sigma(e) \, (- \bar \partial )\, k \\
&=& k^{-1} \, \sigma(e) \, k^2 \, \partial  \, \sigma(e) \,  \bar \partial \, k. 
\end{eqnarray*}

We can now calculate the pseudodifferential symbol of the operator $L^+$. 
\begin{lem}
\label{Lplussymbollem}
We have $\sigma_{L^+}(\xi) = p^+_2(\xi) + p^+_1(\xi) + p^+_0(\xi), $ where 
 \[
 p^+_2(\xi) = k^{-1}\, \sigma(e)\, k^2\,  \sigma(e) \,k \, \left  (  \xi_1^2 +  \xi_2^2 \right ) =k^{-2} e k^2 e k^2 \left  (  \xi_1^2 +  \xi_2^2 \right ),  
 \]
 \begin{eqnarray*}
 p^+_1(\xi) &=& k^{-1}\, \sigma(e)\, k^2\, \Big (  
 \big ( 
 2 \sigma(e) \, \del_1(k) + \del_1(\sigma(e)) k + i \del_2(\sigma(e)) k
 \big ) \xi_1 \\
&& + 
 \big (
 2 \sigma(e) \, \del_2(k) + \del_2(\sigma(e)) k - i \del_1(\sigma(e)) k 
 \big ) \xi_2 \Big ), 
 \end{eqnarray*}
 \begin{eqnarray*}
 p^+_0(\xi)&=& k^{-1} \, \sigma(e) \, k^2 \Big ( \del_1(\sigma(e)) \, \del_1(k) + \sigma(e) \, \delta_1^2(k) + \del_2(\sigma(e)) \, \del_2(k) + \sigma(e)\, \del_2^2(k) \\
 &&+ i \big ( \del_2(\sigma(e))\, \del_1(k)  - \del_1(\sigma(e)) \, \del_2(k)\big  ) \Big ). 
 \end{eqnarray*}
 \begin{proof} 
 It follows from the fact that 
\[
L^+= k^{-1} \, \sigma(e) \, k^2 \, \partial  \, \sigma(e) \,  \bar \partial \, k : \Hc_0 \to \Hc_0, 
\]
and for any $a \in \SNT \subset \Hc_0$ we have 
\begin{eqnarray*}
&&\left ( \partial  \, \sigma(e) \,  \bar \partial \, k\right ) (a) \\
&=& \sigma(e) k \,\del_1^2(a) + \sigma(e) k \, \del_2^2(a) \\
&& + \sigma(e) \del_1(k) \del_1(a) + \del_1(\sigma(e) k) \del_1(a)+\sigma(e) \del_2(k) \del_2(a)+ \del_2(\sigma(e) k) \del_2(a) \\ 
&&+ i \Big (  \sigma(e) \del_1(k) \del_2(a) - \sigma(e) \del_2(k) \del_1(a) + \del_2(\sigma(e) k) \del_1(a)  - \del_1(\sigma(e) k) \del_2(a) \Big ) \\ 
&&+ \del_1(\sigma(e) \del_1(k) )a + \del_2(\sigma(e) \del_2(k) )a+ i \Big ( \del_2(\sigma(e) \del_1(k) )a - \del_1(\sigma(e) \del_2(k) )a\Big ).  
\end{eqnarray*}
 \end{proof}
\end{lem}

Similarly, the operator 
\[
 \De (\De)^*= \sigma(e_-) \pv \iota_+  \iota_+^* \pv^* \sigma(e_-)^*  
= \sigma(e_-) \pv e_+ \pv^* \iota_- : \sigma(e_-) \mathcal{H}^-  \to \sigma(e_-) \mathcal{H}^- 
\]
is equal to the restriction of the operator 
\[ 
L_1^- = \sigma(e_-) \pv e_+ \pv^* : \Hc^- \to \Hc^- 
\]
to the subspace $\sigma(e_-) \Hc^-$. Therefore 
\[ 
L_1^- = 
\left(\begin{array}{c c}
 \De  (\De)^*& * \\
0  & 0 \end{array}\right), 
\]
which yields, 
\[
\exp(-tL_1^-) = 
\left(\begin{array}{c c}
 \exp \left (-t \De  \left (\De \right)^* \right) & * \\
0  & 0 \end{array}\right),
\]
and 
\[
\Tr \left (  \exp(-tL_1^-)  \right ) =  \Tr \left ( \exp \left (-t \De  \left (\De \right)^* \right) \right ). 
\]
The operator $L_1^-: \Hc^- \to \Hc^- $ is anti-unitarily equivalent to the operator 
\[
L^-= J L_1^- J: \Hc_0 \to \Hc_0.
\]

We have: 
\[
L_1^- = R_{\sigma(e)} \, \partial \, R_{\sigma(e)} \, R_{k^2} \, \bar \partial, 
\]
and 
\begin{eqnarray*}
L^- &=&   J L_1^- J \\
&= & J\, R_{\sigma(e)} \, \partial \, R_{\sigma(e)} \, R_{k^2} \, \bar \partial \, J \\
& = & (J\, R_{\sigma(e)} J) \, (J \partial  J)\, (J R_{\sigma(e)} J) \, (JR_{k^2} J)\, (J\bar \partial \, J) \\
&=&  \sigma(e) \, (- \bar \partial) \, \sigma(e) \, k^2 (- \partial) \\ 
&=&  \sigma(e) \, \bar \partial \, \sigma(e) \, k^2 \, \partial. 
\end{eqnarray*}

We can now present the pseudodifferential symbol of $L^-$. 
\begin{lem}
\label{Lminussymbollem}
We have $\sigma_{L^-}(\xi) = p^-_2(\xi) + p^-_1(\xi) $ where 
\[
p^-_2(\xi) = \sigma(e)^2 k^2 (\xi_1^2 + \xi_2^2) = k^{-1} e k^3  (\xi_1^2 + \xi_2^2), 
\]
\begin{eqnarray*}
p^-_1(\xi) = \sigma(e) \Big ( \del_1(\sigma(e) k^2) \xi_1 + \del_2(\sigma(e) k^2) \xi_2  
+ i \big (  \del_1(\sigma(e) k^2) \xi_2 - \del_2(\sigma(e) k^2) \xi_1 \big ) \Big ).  
\end{eqnarray*}
\begin{proof}
It follows from the fact that 
\[
L^-= \sigma(e) \, \bar \partial \, \sigma(e) \, k^2 \partial: \Hc^- \to \Hc^-, 
\]
and for any $a \in \SNT$ we have: 
\begin{eqnarray*}
\left ( \bar \partial \, \sigma(e) \, k^2 \partial \right ) (a) &=&   \sigma(e) k^2 \del_1^2(a) + \sigma(e) k^2 \del_2^2(a) \\
&& + \del_1(\sigma(e) k^2) \del_1(a) + \del_2(\sigma(e) k^2) \del_2 (a)  \\
&& - i  \del_2(\sigma(e) k^2) \del_1 (a) + i  \del_1(\sigma(e) k^2) \del_2 (a). 
\end{eqnarray*}

\end{proof}

\end{lem}
\subsection{Index theorems} 
\label{IndexThmSec}

We dedicate this subsection to review some fundamental techniques 
and results about the heat equation proof of the index theorem,  
which play an important role in our approach in the noncommutative 
setting in the present paper. For a complete account of the details 
one can refer to \cite{MR1031992,  MR1670907,  MR0783634} 
and references therein.

We first explain the McKean-Singer index theorem  \cite{MR0217739}. 
Assume that $M$ is a 
compact manifold of even dimension $m$ and that $V$ and $W$ are hermitian 
vector bundles over $M$. Also let $P: C^\infty(V) \to C^\infty(W)$ be an elliptic differential 
operator from the smooth sections of $V$ to those of $W$, and for simplicity 
assume that $P$ is of order 1. By constructing suitable Sobolev spaces out 
of $V$ and $W$, it is known that $P$ extends to a bounded operator between 
Sobolev spaces, which is a Fredholm operator as well. However, in order to deal with the 
Fredholm index of $P$, one can safely only consider the smooth sections, since 
the {\em regularity} stemming from the ellipticity of $P$ ensures that all the 
sections in the kernel of (the extension of) $P$ are smooth, see for example 
\cite{ MR0783634}. Then, in order to calculate the index,  
\[
\textnormal{Ind}(P) = \textnormal{Dim Ker}(P) - \textnormal{Dim Ker}(P^*), 
\]
where $P^*$ is the adjoint of $P$, one can use the heat expansion as follows. 
The McKean-Singer index theorem states that for any $t>0:$
\begin{equation} \label{MCKeanSingerFormula}
\textnormal{Ind}(P) 
=
\textnormal{Trace}\left (  \exp(-t P^* P) \right ) 
- 
\textnormal{Trace}\left (  \exp(-t P P^*) \right ).  
\end{equation}

On the other hand there are small time asymptotic expansions of the form \eqref{asympexp}, 
which depend on the local symbols. 
That is,  there are densities $a_{2j}(x, P^*P) \, dx$ and $a_{2j}(x, PP^*) \, dx$ on $M$ obtained locally 
from the pseudodifferenital symbols of $P^*P$ and $PP^*$  such that 
\begin{equation} \label{asympexp2}
\textnormal{Trace} \left ( \exp(-t P^*P) \right )  \,\,\, 
\sim_{t \to 0^+}  
\,\,\, 
t^{- m/2}  
\sum_{j=0}^\infty t^j \int_M  a_{2j}(x, P^*P) \, dx, 
\end{equation} 
and 
\begin{equation} \label{asympexp3}
\textnormal{Trace} \left ( \exp(-t PP^*) \right )  \,\,\, 
\sim_{t \to 0^+}  
\,\,\, 
t^{- m/2}  
\sum_{j=0}^\infty t^j \int_M a_{2j}(x, PP^*) \, dx. 
\end{equation} 
Since the index of $P$ in \eqref{MCKeanSingerFormula} is independent of $t$, 
after writing the expansions in the right hand side using \eqref{asympexp2}, 
\eqref{asympexp3}, the only term that is independent of $t$ turns out to match with 
the index, hence: 
\begin{equation}
\label{localformula0}
\textnormal{Ind}(P) 
= \int_M \left ( a_{m}(x, P^* P) - a_{m}(x, PP^*)  \right )\, dx 
\end{equation}
This shows that there is a local formula for the index, and the celebrated Atiyah-Sinder 
index theorem identifies the density that integrates to the index in terms of characteristic 
classes when $P$ is an important geometric operator. For example when $P$ is the 
de Rham operator $d+d^*$ mapping the even differential forms to the odd ones, the index 
is the Euler characteristic of the manifold, which is equal to the integral of the 
pfaffian of the matrix of curvature 2-forms.  In general, the index theorem states that if $D_E$ is the 
Dirac operator of a spin structure $D$ with coefficients in a vector bundle $E$, then 
the $\hat{ \mathcal{A}}$-class of the tangent bundle $TM$ and the Chern character 
of $E$ give the index by the formula
\begin{equation}
\textnormal{Ind}(D_E) = \int_M \hat{ \mathcal{A}}(TM) \, ch(E). 
\end{equation}
For various proofs of the index theorem one can refer to  
\cite{MR0157392, MR0236950, MR0236951, MR0236952, MR1303779, MR1817560, MR0836727, MR0744920, MR0934249}, 
and references therein.

The elliptic theory of differential operators and pseudodifferential calculus operate perfectly in noncommutative settings as developed in \cite{MR0572645}, and analyzed in 
further detail for noncommutative tori in 
\cite{2018JPhCS.965a2042T, 2018arXiv180303575H, 2018arXiv180303580H}.  
Moreover, there is a  crucial need for local index formulas 
for twisted spectral triples as we explained in \S \ref{SpecTripSec}. Therefore, in the present 
paper we take the heat expansion approach to find a local formula for the index 
of the Dirac operator of the twisted spectral triple described in \S \ref{NCconformalsec} 
with coefficients (or twisted by) an auxiliary finitely generated projective module 
on the noncommutative two torus, playing the role of a general vector bundle, 
cf. \cite{MR0587369, MR0703981}.  
We follow closely the setup provided in \cite{MR2427588, MR3500845}, 
and the work has intimate connections with \cite{MR3540454}.

\section{Calculation of a noncommutative local formula for the index}
\label{calculationssec}

In this section we apply the method explained in \S \ref{IndexThmSec},  
for the calculation of a local formula 
for the index of a twisted Dirac operator, 
to the twisted Dirac operator, 
\[
D_{e, \sigma}^+ = \sigma(e_-) \partial_\vphi e_+ : e_+ \mathcal{H}^+ 
\to
\sigma(e_-) \mathcal{H}^-,
\]
of a conformally flat metric on noncommutative two torus $\NT$. 
Based on the discussion following the McKean-Singer formula 
\eqref{MCKeanSingerFormula}, our treatment in \S \ref{Lplusminussec} 
of the operators $(D_{e, \sigma}^+)^* D_{e, \sigma}^+$ and 
$D_{e, \sigma}^+ (D_{e, \sigma}^+)^*$ to 
derive their anti-unitarily equivalents $L^+$ and $L^-$, and 
the formula \eqref{asmpexpnct} for the terms in the heat expansion, we have: 
\begin{equation} 
\label{indDplus}
{\rm Ind}(D_{e, \sigma}^+) 
= 
\tau \left ( a_2 \left ( L^+\right ) 
- 
a_2 \left (  L^- \right) \right ). 
\end{equation}
Thus, our task is now to perform individually for $L^+$ and $L^-$ the recursive procedure that leads to 
the explicit formula \eqref{a_2nformula} and to derive explicit formulas for $a_2 \left ( L^+\right ) $ and 
$a_2 \left ( L^-\right ) $. 

Using the homogeneous components of the pseudodifferential symbols of 
$L^{\pm}$  presented in Lemmas \ref{Lplussymbollem} and \ref{Lminussymbollem}, we perform 
symbolic calculations and use formula \eqref{r_nformula} to calculate 
$b_2(\xi, \lambda, L^{\pm})$. This leads to  lengthy expressions which we calculate 
with computer assistance. The next task is to use \eqref{a_2nformula} to derive 
$a_2(L^{\pm})$, namely 
\[
a_{2}(L^{\pm}) = \frac{1}{2\pi i}  \int_{\R^2} \int_\gamma e^{-\lambda} b_{2}(\xi, \lambda, L^\pm)  \, d\lambda \,d\xi. 
\]
Using a homogeneity argument, see \cite{MR2907006, MR3194491}, one can 
avoid the contour integration in the latter formula by setting $\lambda = -1$ and multiplying 
the final result result by $-1$. More precisely, one has: 
\begin{equation}
\label{a_2Lplusminus}
a_{2}(L^{\pm}) = - \int_{\R^2} b_{2}(\xi, -1, L^\pm)  \,d\xi. 
\end{equation}
Using the trace property of $\tau$ appearing in the formula for the index 
\eqref{indDplus}, after calculating  $b_{2}(\xi, -1, L^\pm)$, we will 
rotate cyclically the multiplicative term $b_0(\xi, -1, L^\pm)$ appearing 
at the very right side of our terms and bring it to the very left side.  We then 
carry out the integration over $\R^2$ by passing to the polar coordinates 
$\xi_1 = r \cos \theta$, $\xi_2 = r \sin \theta$. The angular integration with respect to 
$\theta$, from $0$ to $2 \pi$, can be done in a straight forward manner. However, 
the radial integration with respect to $r$, from $0$ to $\infty$, poses a challenge 
coming from the noncommutativity. This challenge appeared in \cite{MR2907006, MR2947960, MR3194491, MR3148618} as well and was overcome 
by the rearrangement lemma and led to the appearance of the modular automorphism in final formulas. 
In this article we need an even more elaborate version of the rearrangement lemma due to the presence of 
an idempotent in addition to a conformal factor, as we shall see shortly.

\subsection{Computation of $ \tau \left ( a_2 \left ( L^- \right ) \right )$} 
Considering the symbol of $L^-$ written in Lemma \ref{Lminussymbollem} 
and using the recursive formulas \eqref{r_0formula} and \eqref{r_nformula}, we have 
\[
b_0(\xi, -1, L^-)=(\sigma(e)k^2|\xi|^2+1)^{-1}, 
\] 
and up to a right multiplication by $-b_0(\xi, -1, L^-)$, $b_2(\xi, -1, L^-)$ is the sum of 395 terms: 
\begin{center}
\begin{math}
b_2(\xi, -1, L^-)= 
-2\xi_1^2\left(b_0^2\sigma(e)k^2\right)\sigma(e)k\delta_1^2(k)+2i\xi_1^2\left(b_0^2\sigma(e)k^2\right)\sigma(e)k\left(\delta_1\delta_2(k)\right)-4\xi_1^2
\left(b_0^2\sigma(e)k^2\right)\sigma(e)\delta_1(k)\delta_1(k)+2i\xi_1^2\left(b_0^2\sigma(e)k^2\right)\sigma(e)\delta_1(k)\delta_2(k)-2\xi_1^2\left(b_0^2
\sigma(e)k^2\right)\sigma(e)\delta_1^2(k)k+2i\xi_1^2\left(b_0^2\sigma(e)k^2\right)\sigma(e)\delta_2(k)\delta_1(k)+2i\xi_1^2\left(b_0^2
\sigma(e)k^2\right)\sigma(e)\left(\delta_1\delta_2(k)\right)k-\xi_1^2\left(b_0^2\sigma(e)k^2\right)\delta_1^2(\sigma(e)k^2)-\cdots
\end{math}
\end{center}
Passing to polar coordinates and integrating the angular variable, we get terms of the following form,
up to a right multiplication by $-2\pi b_0(\xi, -1, L^-)$, 
\begin{center}
\begin{math}
-2\left(\sigma(e)k^2b_0^2\right)\delta_1(\sigma(e)k^2)\left(\sigma(e)k^2b_0^2\right)\delta_1(\sigma(e)k^2)r^6-2\left(\sigma(e)k^2b_0^2\right)\delta
_2(\sigma(e)k^2)\left(\sigma(e)k^2b_0^2\right)\delta_2(\sigma(e)k^2)r^6-4\left((\sigma(e)k^2)^2b_0^3\right)\delta_1(\sigma(e)k^2)b_0\delta_1(\sigma(e)k^2)r^6-4
\left((\sigma(e)k^2)^2b_0^3\right)\delta_2(\sigma(e)k^2)b_0\delta_2(\sigma(e)k^2)r^6+2\left((\sigma(e)k^2)^2b_0^3\right)\delta_1^2(\sigma(e)k^2)r^4+\cdots, 
\end{math}
\end{center}
a total of 82 terms of the form. The next step is to perform 
 the integration against $r\,\mathrm dr$ over $0$ to $\infty$.

An important integration formula we later prove requires that the leftmost powers of $b_0$ have
a factor of $\sigma(e)$ to their right. Leftmost $b_0$'s raised only to the power 1 already have this factor $\sigma(e)$
immediately to their right, and leftmost $b_0$'s raised to powers 2 or 3 have a power of $\sigma(e)k^2$
to the left which we can move to the right, again giving us  $\sigma(e)$ immediately to the right. As we explained earlier, 
It is also useful to use the trace property of $\tau$ to permute the multiplicative factors cyclically. 
Thus, instead of multiplying the above sum by $-2\pi b_0(\xi, -1, L^-)$. 

\subsubsection{Terms with all $b_0$ on the left}
As it turns out, it is difficult to evaluate improper integrals when nontrivial idempotents are involved. 
For instance, consider the integral $\int_0^{\infty}\!e/(1+ex)^2\,\mathrm dx$. One would expect that
$$\int_0^{\infty}\!\frac{e\,\mathrm dx}{(1+ex)^2}=\left[\frac{-1}{1+ex}\right]_0^{\infty}=1,$$
but that implies that
$$e=e\cdot 1=e\cdot\int_0^{\infty}\!\frac{e\,\mathrm dx}{(1+ex)^2}=\int_0^{\infty}\!\frac{e\,\mathrm dx}{(1+ex)^2}=1,$$
giving us a contradiction.  Since
$$e(1-ez+(ez)^2-\cdots)^2=e(1-z+z^2-\cdots)^2$$
for $|z|<1$, we have
$$\frac{e}{(1+ez)^2}=\frac{e}{(1+z)^2},$$
so we can properly evaluate the aforementioned improper integral as follows:
$$\int_0^{\infty}\!\frac{e\,\mathrm dx}{(1+ex)^2}=e\int_0^{\infty}\!\frac{\,\mathrm dx}{(1+x)^2}
=e\left[\frac{-1}{1+x}\right]_0^{\infty}=e.$$
To calculate the contribution of terms with all $b_0$ on the left to the trace, we need a formula for
$$\int_0^{\infty}\!(\sigma(e)k^2u+1)^{m+2}\backslash u^m\sigma(e)\,\mathrm du.$$
Since  $\sigma(e)$ and $k$ do not necessarily commute,  we cannot do the trick we did above. We need to reduce
to the case of terms with $b_0$ in the middle. The trace property allows to write:
\begin{eqnarray*}
&&\tau\left(\int_0^{\infty}\!(\sigma(e)k^2u+1)^{m+2}\backslash u^m\sigma(e)\rho\,\mathrm du\right) \\
&=&\tau\left(\int_0^{\infty}\!(\sigma(e)k^2u+1)^{m+1}\backslash(u^m\sigma(e)\rho)/(\sigma(e)k^2u+1)\,\mathrm du\right).
\end{eqnarray*}
Even though we get what looks like a more difficult integral, we can use techniques from combinatorics, Fourier analysis, and complex analysis to evaluate it, inspired by a proof given in \cite{MR2907006} for the case $e=1$.

\subsubsection{Terms with $b_0^2$ in the middle}
As in \cite{MR2907006}, we use integration by parts so that we can write terms with $b_0^2$ as terms with $b_0$ in the middle.
The main difference is that this time, $b_0=(\sigma(e)k^2u+1)^{-1}$, so $\partial_r(b_0)=-2\sigma(e)k^2rb_0^2$ and
instead of a term like
$$\int_0^{\infty}r^6(k^2b_0^3)\delta_1(k)k(k^2b_0^2)\delta_1(k)kr\,\mathrm dr,$$
we would have something like
$$\int_0^{\infty}r^6(\sigma(e)k^2b_0^3)\delta_1(k)k(\sigma(e)k^2b_0^2)\delta_1(k)kr\,\mathrm dr,$$
which we would replace by
$$\frac{\sigma(e)k^2}{2}\int_0^{\infty}\partial_r(r^6b_0^3)\delta_1(k)kb_0\delta_1(k)k\,\mathrm dr.$$
Again, we are able to reduce to the case of terms with $b_0$ in the middle.

\subsubsection{Terms with $b_0$ in the middle}
For integrals involving elements of $\SNT$ squeezed between powers
of $b_0=b_0(\xi, -1, L^-)$, we need to prove a rearrangement lemma. In the most basic case,
which suffices for our needs, the idempotent appears with $k^2$ in the denominator of
the $b_0$, and we need the following generalized version of the
rearragement lemma \cite[Lemma 6.2]{MR2907006}:
\begin{lem}\label{mvt}
For every element $\rho\in A_{\theta}^{\infty}$ and every non-negative integer $m$ we have
$$
\int_0^{\infty}\!(\sigma(ek^2)u+1)^{m+1}\backslash(u^m\sigma(e)\rho)/(\sigma(ek^2)u+1)\,\mathrm du
=
\sigma(e) D_m(k^{-(2m+2)}\sigma(e)\rho \sigma(e)),
$$
where  
$$ 
D_m=\mathcal L_m(\Delta), 
$$
\begin{align*}
\mathcal L_m(u)&=\int_0^{\infty}\!\frac{x^m}{(x+1)^{m+1}}\frac{1}{xu+1}\,\mathrm dx \\
&=(-1)^m(u-1)^{-(m+1)}\left(\log u-\sum_{j=1}^m(-1)^{j+1}\frac{(u-1)^j}{j}\right), 
\end{align*}
$\Delta(a)=k^{-2}ak^2$, $a \in \CNT$, is the modular automorphism, 
and $\sigma(a)=k^{-1}ak$ is its square root.
\end{lem}
\begin{proof} The proof is provided in  Appendix \ref{appproofsec}. 
\end{proof}
The following corollary will be useful in our calculations: 
\begin{cor}\label{mvc}
\begin{align*}
&\int_0^{\infty}\!(\Delta(ek^2e)u+1)^{m+1}\backslash(u^m\Delta(e)\rho)/(\Delta(ek^2e)u+1)\,\mathrm du \\
&=\Delta(e) D_m(k^{-(2m+2)}\Delta(e)\rho \Delta(e))\Delta(e)
\end{align*}
\end{cor}
\begin{proof}
The proof is provided in  Appendix \ref{appproofsec}. 
\end{proof}

Applying Lemma \ref{mvt} and using the Leibniz rule to simplify,
we find that
\begin{center}
\begin{math}
\frac{1}{-2\pi}\,\tau \left ( a_2\left (L^- \right) \right )
=\tau \Big (\frac{1}{2}\sigma(e)D_1\left(k^{-4}\sigma(e)k\delta_1(k)\sigma(e)\right)\delta_1(\sigma(e)k^2)+\frac{1}{2}i\sigma(e)
D_1\left(k^{-4}\sigma(e)k\delta_1(k)\sigma(e)\right)\delta_2(\sigma(e)k^2)-\frac{1}{2}i\sigma(e)
D_1\left(k^{-4}\sigma(e)k\delta_2(k)\sigma(e)\right)\delta_1(\sigma(e)k^2)+\frac{1}{2}\sigma(e)
D_1\left(k^{-4}\sigma(e)k\delta_2(k)\sigma(e)\right)\delta_2(\sigma(e)k^2)+\frac{1}{2}\sigma(e)
D_1\left(k^{-4}\sigma(e)\delta_1(k)k\sigma(e)\right)\delta_1(\sigma(e)k^2)+\frac{1}{2}i\sigma(e)
D_1\left(k^{-4}\sigma(e)\delta_1(k)k\sigma(e)\right)\delta_2(\sigma(e)k^2)-\frac{1}{2}i\sigma(e)
D_1\left(k^{-4}\sigma(e)\delta_2(k)k\sigma(e)\right)\delta_1(\sigma(e)k^2)+\frac{1}{2}\sigma(e)
D_1\left(k^{-4}\sigma(e)\delta_2(k)k\sigma(e)\right)\delta_2(\sigma(e)k^2)+\frac{1}{2}\sigma(e)
D_1\left(k^{-4}\sigma(e)\delta_1(\sigma(e))k^2\sigma(e)\right)\delta_1(\sigma(e)k^2)+\frac{1}{2}i\sigma(e)
D_1\left(k^{-4}\sigma(e)\delta_1(\sigma(e))k^2\sigma(e)\right)\delta_2(\sigma(e)k^2)-\frac{1}{2}i\sigma(e)
D_1\left(k^{-4}\sigma(e)\delta_2(\sigma(e))k^2\sigma(e)\right)\delta_1(\sigma(e)k^2)+\frac{1}{2}\sigma(e)
D_1\left(k^{-4}\sigma(e)\delta_2(\sigma(e))k^2\sigma(e)\right)\delta_2(\sigma(e)k^2)-\sigma(e)k^2\sigma(e)
D_2\left(k^{-4}\sigma(e)k\delta_1(k)\sigma(e)\right)\delta_1(\sigma(e)k^2)-i\sigma(e)k^2\sigma(e)
D_2\left(k^{-4}\sigma(e)k\delta_1(k)\sigma(e)\right)\delta_2(\sigma(e)k^2)+i\sigma(e)k^2\sigma(e)
D_2\left(k^{-4}\sigma(e)k\delta_2(k)\sigma(e)\right)\delta_1(\sigma(e)k^2)-\sigma(e)k^2\sigma(e)
D_2\left(k^{-4}\sigma(e)k\delta_2(k)\sigma(e)\right)\delta_2(\sigma(e)k^2)-\sigma(e)k^2\sigma(e)
D_2\left(k^{-4}\sigma(e)\delta_1(k)k\sigma(e)\right)\delta_1(\sigma(e)k^2)-i\sigma(e)k^2\sigma(e)
D_2\left(k^{-4}\sigma(e)\delta_1(k)k\sigma(e)\right)\delta_2(\sigma(e)k^2)+i\sigma(e)k^2\sigma(e)
D_2\left(k^{-4}\sigma(e)\delta_2(k)k\sigma(e)\right)\delta_1(\sigma(e)k^2)-\sigma(e)k^2\sigma(e)
D_2\left(k^{-4}\sigma(e)\delta_2(k)k\sigma(e)\right)\delta_2(\sigma(e)k^2)-\sigma(e)k^2\sigma(e)
D_2\left(k^{-4}\sigma(e)\delta_1(\sigma(e))k^2\sigma(e)\right)\delta_1(\sigma(e)k^2)-i\sigma(e)k^2\sigma(e)
D_2\left(k^{-4}\sigma(e)\delta_1(\sigma(e))k^2\sigma(e)\right)\delta_2(\sigma(e)k^2)+i\sigma(e)k^2\sigma(e)
D_2\left(k^{-4}\sigma(e)\delta_2(\sigma(e))k^2\sigma(e)\right)\delta_1(\sigma(e)k^2)-\sigma(e)k^2\sigma(e)
D_2\left(k^{-4}\sigma(e)\delta_2(\sigma(e))k^2\sigma(e)\right)\delta_2(\sigma(e)k^2)-\sigma(e)
D_2\left(k^{-6}\sigma(e)k^2\delta_1(\sigma(e)k^2)\sigma(e)\right)\delta_1(\sigma(e)k^2)-\sigma(e)
D_2\left(k^{-6}\sigma(e)k^2\delta_2(\sigma(e)k^2)\sigma(e)\right)\delta_2(\sigma(e)k^2)+\sigma(e)
D_2\left(k^{-6}\sigma(e)k^2\sigma(e)k\delta_1(k)\sigma(e)\right)\delta_1(\sigma(e)k^2)+i\sigma(e)
D_2\left(k^{-6}\sigma(e)k^2\sigma(e)k\delta_1(k)\sigma(e)\right)\delta_2(\sigma(e)k^2)-i\sigma(e)
D_2\left(k^{-6}\sigma(e)k^2\sigma(e)k\delta_2(k)\sigma(e)\right)\delta_1(\sigma(e)k^2)+\sigma(e)
D_2\left(k^{-6}\sigma(e)k^2\sigma(e)k\delta_2(k)\sigma(e)\right)\delta_2(\sigma(e)k^2)+\sigma(e)
D_2\left(k^{-6}\sigma(e)k^2\sigma(e)\delta_1(k)k\sigma(e)\right)\delta_1(\sigma(e)k^2)+i\sigma(e)
D_2\left(k^{-6}\sigma(e)k^2\sigma(e)\delta_1(k)k\sigma(e)\right)\delta_2(\sigma(e)k^2)-i\sigma(e)
D_2\left(k^{-6}\sigma(e)k^2\sigma(e)\delta_2(k)k\sigma(e)\right)\delta_1(\sigma(e)k^2)+\sigma(e)
D_2\left(k^{-6}\sigma(e)k^2\sigma(e)\delta_2(k)k\sigma(e)\right)\delta_2(\sigma(e)k^2)+\sigma(e)
D_2\left(k^{-6}\sigma(e)k^2\sigma(e)\delta_1(\sigma(e))k^2\sigma(e)\right)\delta_1(\sigma(e)k^2)+i\sigma(e)
D_2\left(k^{-6}\sigma(e)k^2\sigma(e)\delta_1(\sigma(e))k^2\sigma(e)\right)\delta_2(\sigma(e)k^2)-i\sigma(e)
D_2\left(k^{-6}\sigma(e)k^2\sigma(e)\delta_2(\sigma(e))k^2\sigma(e)\right)\delta_1(\sigma(e)k^2)+\sigma(e)
D_2\left(k^{-6}\sigma(e)k^2\sigma(e)\delta_2(\sigma(e))k^2\sigma(e)\right)\delta_2(\sigma(e)k^2)+4\sigma(e)k^2\sigma(e)
D_3\left(k^{-6}\sigma(e)k^2\delta_1(\sigma(e)k^2)\sigma(e)\right)\delta_1(\sigma(e)k^2)+4\sigma(e)k^2\sigma(e)
D_3\left(k^{-6}\sigma(e)k^2\delta_2(\sigma(e)k^2)\sigma(e)\right)\delta_2(\sigma(e)k^2)-4\sigma(e)
D_3\left(k^{-8}\sigma(e)k^2\sigma(e)k^2\delta_1(\sigma(e)k^2)\sigma(e)\right)\delta_1(\sigma(e)k^2)-4
\sigma(e)D_3\left(k^{-8}\sigma(e)k^2\sigma(e)k^2\delta_2(\sigma(e)k^2)\sigma(e)\right)\delta
_2(\sigma(e)k^2)+\sigma(e)D_2\left(k^{-6}\sigma(e)k^2\delta_1(\sigma(e)k^2)\sigma(e)\right)\sigma(e)k\delta_1(k)-i\sigma(e)
D_2\left(k^{-6}\sigma(e)k^2\delta_1(\sigma(e)k^2)\sigma(e)\right)\sigma(e)k\delta_2(k)+\sigma(e)
D_2\left(k^{-6}\sigma(e)k^2\delta_1(\sigma(e)k^2)\sigma(e)\right)\sigma(e)\delta_1(k)k-i\sigma(e)
D_2\left(k^{-6}\sigma(e)k^2\delta_1(\sigma(e)k^2)\sigma(e)\right)\sigma(e)\delta_2(k)k+i\sigma(e)
D_2\left(k^{-6}\sigma(e)k^2\delta_2(\sigma(e)k^2)\sigma(e)\right)\sigma(e)k\delta_1(k)+\sigma(e)
D_2\left(k^{-6}\sigma(e)k^2\delta_2(\sigma(e)k^2)\sigma(e)\right)\sigma(e)k\delta_2(k)+i\sigma(e)
D_2\left(k^{-6}\sigma(e)k^2\delta_2(\sigma(e)k^2)\sigma(e)\right)\sigma(e)\delta_1(k)k+\sigma(e)
D_2\left(k^{-6}\sigma(e)k^2\delta_2(\sigma(e)k^2)\sigma(e)\right)\sigma(e)\delta_2(k)k+\sigma(e)
D_2\left(k^{-6}\sigma(e)k^2\delta_1(\sigma(e)k^2)\sigma(e)\right)\sigma(e)\delta_1(\sigma(e))k^2-i\sigma(e)
D_2\left(k^{-6}\sigma(e)k^2\delta_1(\sigma(e)k^2)\sigma(e)\right)\sigma(e)\delta_2(\sigma(e))k^2+i\sigma(e)
D_2\left(k^{-6}\sigma(e)k^2\delta_2(\sigma(e)k^2)\sigma(e)\right)\sigma(e)\delta_1(\sigma(e))k^2+\sigma(e)
D_2\left(k^{-6}\sigma(e)k^2\delta_2(\sigma(e)k^2)\sigma(e)\right)\sigma(e)\delta_2(\sigma(e))k^2-\sigma(e)D_1\left(k^{-4}\sigma(e)k^2\delta_1^2(\sigma(e)k^2)\sigma(e)\right)-\sigma(e)
D_1\left(k^{-4}\sigma(e)k^2\delta_2^2(\sigma(e)k^2)\sigma(e)\right)-\sigma(e)
D_1\left(k^{-4}\sigma(e)k^2\sigma(e)k\delta_1^2(k)\sigma(e)\right)-\sigma(e)
D_1\left(k^{-4}\sigma(e)k^2\sigma(e)k\delta_2^2(k)\sigma(e)\right)-2\sigma(e)
D_1\left(k^{-4}\sigma(e)k^2\sigma(e)\delta_1(k)\delta_1(k)\sigma(e)\right)-\sigma(e)
D_1\left(k^{-4}\sigma(e)k^2\sigma(e)\delta_1^2(k)k\sigma(e)\right)-2\sigma(e)
D_1\left(k^{-4}\sigma(e)k^2\sigma(e)\delta_2(k)\delta_2(k)\sigma(e)\right)-\sigma(e)
D_1\left(k^{-4}\sigma(e)k^2\sigma(e)\delta_2^2(k)k\sigma(e)\right)-2\sigma(e)
D_1\left(k^{-4}\sigma(e)k^2\sigma(e)\delta_1(\sigma(e))k\delta_1(k)\sigma(e)\right)-2\sigma(e)
D_1\left(k^{-4}\sigma(e)k^2\sigma(e)\delta_1(\sigma(e))\delta_1(k)k\sigma(e)\right)-\sigma(e)
D_1\left(k^{-4}\sigma(e)k^2\sigma(e)\delta_1^2(\sigma(e))k^2\sigma(e)\right)-2\sigma(e)
D_1\left(k^{-4}\sigma(e)k^2\sigma(e)\delta_2(\sigma(e))k\delta_2(k)\sigma(e)\right)-2\sigma(e)
D_1\left(k^{-4}\sigma(e)k^2\sigma(e)\delta_2(\sigma(e))\delta_2(k)k\sigma(e)\right)-\sigma(e)
D_1\left(k^{-4}\sigma(e)k^2\sigma(e)\delta_2^2(\sigma(e))k^2\sigma(e)\right)-\sigma(e)
D_1\left(k^{-4}\sigma(e)k^2\delta_1(\sigma(e))\sigma(e)k\delta_1(k)\sigma(e)\right)+i\sigma(e)
D_1\left(k^{-4}\sigma(e)k^2\delta_1(\sigma(e))\sigma(e)k\delta_2(k)\sigma(e)\right)-\sigma(e)
D_1\left(k^{-4}\sigma(e)k^2\delta_1(\sigma(e))\sigma(e)\delta_1(k)k\sigma(e)\right)+i\sigma(e)
D_1\left(k^{-4}\sigma(e)k^2\delta_1(\sigma(e))\sigma(e)\delta_2(k)k\sigma(e)\right)-\sigma(e)
D_1\left(k^{-4}\sigma(e)k^2\delta_1(\sigma(e))\delta_1(\sigma(e))k^2\sigma(e)\right)+i\sigma(e)
D_1\left(k^{-4}\sigma(e)k^2\delta_1(\sigma(e))\delta_2(\sigma(e))k^2\sigma(e)\right)-i\sigma(e)
D_1\left(k^{-4}\sigma(e)k^2\delta_2(\sigma(e))\sigma(e)k\delta_1(k)\sigma(e)\right)-\sigma(e)
D_1\left(k^{-4}\sigma(e)k^2\delta_2(\sigma(e))\sigma(e)k\delta_2(k)\sigma(e)\right)-i\sigma(e)
D_1\left(k^{-4}\sigma(e)k^2\delta_2(\sigma(e))\sigma(e)\delta_1(k)k\sigma(e)\right)-\sigma(e)
D_1\left(k^{-4}\sigma(e)k^2\delta_2(\sigma(e))\sigma(e)\delta_2(k)k\sigma(e)\right)-i\sigma(e)
D_1\left(k^{-4}\sigma(e)k^2\delta_2(\sigma(e))\delta_1(\sigma(e))k^2\sigma(e)\right)-\sigma(e)
D_1\left(k^{-4}\sigma(e)k^2\delta_2(\sigma(e))\delta_2(\sigma(e))k^2\sigma(e)\right)+2\sigma(e)
D_2\left(k^{-6}\sigma(e)k^2\sigma(e)k^2\delta_1^2(\sigma(e)k^2)\sigma(e)\right)+2\sigma(e)
D_2\left(k^{-6}\sigma(e)k^2\sigma(e)k^2\delta_2^2(\sigma(e)k^2)\sigma(e)\right) \Big ).
\end{math}
\end{center}

\subsection{Computation of $\tau \left  ( a_2 \left ( L^+ \right ) \right )$}
In a very similar manner, starting from the symbol of the operator $L^+$ written in Lemma 
\ref{Lplussymbollem}, we calculate a local formula for $\tau \left  ( a_2 \left ( L^+ \right ) \right )$. 
We find that $b_2(\xi, -1, L^+)$, up to right multiplication by $-b_0(\xi, -1, L^+)$, is the sum of 232 terms 
of the following form:
\begin{center}
\begin{math}
\left((b_0')\Delta(e)\right)k^{-1}\sigma(e)k^2\sigma(e)\delta_1^2(k)+\left((b_0')\Delta(e)\right)k^{-1}\sigma(e)k^2\sigma(e)\delta_2^2(k)+\left((b_0')
\Delta(e)\right)k^{-1}\sigma(e)k^2\delta_1(\sigma(e))\delta_1(k)-i\left((b_0')\Delta(e)\right)k^{-1}\sigma(e)k^2\delta_1(\sigma(e))\delta_2(k)+i
\left((b_0')\Delta(e)\right)k^{-1}\sigma(e)k^2\delta_2(\sigma(e))\delta_1(k)+\left((b_0')\Delta(e)\right)k^{-1}\sigma(e)k^2\delta_2(\sigma(e))\delta
_2(k)-4\xi_1^2\left((b_0')^2\Delta(e)k^2\right)k^{-1}\sigma(e)k^2\sigma(e)\delta_1^2(k)
-\cdots.
\end{math}
\end{center}
After passing to the polar coordinates and performing the angular integration, up to write multiplication by  
$-2\pi b_0(\xi, -1, L^+)$, these terms add up to the sum of 82 terms:
\begin{center}
\begin{math}
-2\left(\Delta(e)k^2(b_0')^2\Delta(e)\right)\delta_1(\Delta(e)k^2\Delta(e))\left(\Delta(e)k^2(b_0')^2\Delta(e)\right)\delta_1(\Delta(e)k^2\Delta(e))r^6-2\left(\Delta(e)k^2
(b_0')^2\Delta(e)\right)\delta_2(\Delta(e)k^2\Delta(e))\left(\Delta(e)k^2(b_0')^2\Delta(e)\right)\delta_2(\Delta(e)k^2\Delta(e))r^6-4\left((\Delta(e)k^2)^2(b_0')^3
\Delta(e)\right)\delta_1(\Delta(e)k^2\Delta(e))\left((b_0')\Delta(e)\right)\delta_1(\Delta(e)k^2\Delta(e))r^6-4\left((\Delta(e)k^2)^2(b_0')^3\Delta(e)\right)\delta
_2(\Delta(e)k^2\Delta(e))\left((b_0')\Delta(e)\right)\delta_2(\Delta(e)k^2\Delta(e))r^6+\cdots
\end{math}
\end{center}
Applying Corollary \ref{mvc}, which allows replacing $b_0=b_0(\xi, -1, L^+)$ with $b_0 \Delta(e)$, 
we find that: 
\begin{center}
\begin{math}
\frac{1}{-2\pi}\,\tau \left ( a_2\left (L^+ \right) \right )=\tau \Big (\Delta(e)D_1\left(k^{-4}\Delta(e)k\sigma(e)\delta_1(k)\Delta(e)\right)\Delta(e)\delta_1(\Delta(e)k^2\Delta(e))+\Delta(e)
D_1\left(k^{-4}\Delta(e)k\sigma(e)\delta_2(k)\Delta(e)\right)\Delta(e)\delta_2(\Delta(e)k^2\Delta(e))+\frac{1}{2}\Delta(e)
D_1\left(k^{-4}\Delta(e)k\delta_1(\sigma(e))k\Delta(e)\right)\Delta(e)\delta_1(\Delta(e)k^2\Delta(e))-\frac{1}{2}i\Delta(e)
D_1\left(k^{-4}\Delta(e)k\delta_1(\sigma(e))k\Delta(e)\right)\Delta(e)\delta_2(\Delta(e)k^2\Delta(e))+\frac{1}{2}i\Delta(e)
D_1\left(k^{-4}\Delta(e)k\delta_2(\sigma(e))k\Delta(e)\right)\Delta(e)\delta_1(\Delta(e)k^2\Delta(e))+\frac{1}{2}\Delta(e)
D_1\left(k^{-4}\Delta(e)k\delta_2(\sigma(e))k\Delta(e)\right)\Delta(e)\delta_2(\Delta(e)k^2\Delta(e))-2\Delta(e)k^2\Delta(e)
D_2\left(k^{-4}\Delta(e)k\sigma(e)\delta_1(k)\Delta(e)\right)\Delta(e)\delta_1(\Delta(e)k^2\Delta(e))-2\Delta(e)k^2\Delta(e)
D_2\left(k^{-4}\Delta(e)k\sigma(e)\delta_2(k)\Delta(e)\right)\Delta(e)\delta_2(\Delta(e)k^2\Delta(e))-\Delta(e)k^2\Delta(e)
D_2\left(k^{-4}\Delta(e)k\delta_1(\sigma(e))k\Delta(e)\right)\Delta(e)\delta_1(\Delta(e)k^2\Delta(e))+i\Delta(e)k^2\Delta(e)
D_2\left(k^{-4}\Delta(e)k\delta_1(\sigma(e))k\Delta(e)\right)\Delta(e)\delta_2(\Delta(e)k^2\Delta(e))-i\Delta(e)k^2\Delta(e)
D_2\left(k^{-4}\Delta(e)k\delta_2(\sigma(e))k\Delta(e)\right)\Delta(e)\delta_1(\Delta(e)k^2\Delta(e))-\Delta(e)k^2\Delta(e)
D_2\left(k^{-4}\Delta(e)k\delta_2(\sigma(e))k\Delta(e)\right)\Delta(e)\delta_2(\Delta(e)k^2\Delta(e))-\Delta(e)
D_2\left(k^{-6}\Delta(e)k^2\Delta(e)\delta_1(\Delta(e)k^2\Delta(e))\Delta(e)\right)\Delta(e)\delta
_1(\Delta(e)k^2\Delta(e))-\Delta(e)D_2\left(k^{-6}\Delta(e)k^2\Delta(e)\delta
_2(\Delta(e)k^2\Delta(e))\Delta(e)\right)\Delta(e)\delta_2(\Delta(e)k^2\Delta(e))+2\Delta(e)
D_2\left(k^{-6}\Delta(e)k^2\Delta(e)k\sigma(e)\delta_1(k)\Delta(e)\right)\Delta(e)\delta_1(\Delta(e)k^2\Delta(e))+2
\Delta(e)D_2\left(k^{-6}\Delta(e)k^2\Delta(e)k\sigma(e)\delta_2(k)\Delta(e)\right)\Delta(e)\delta
_2(\Delta(e)k^2\Delta(e))+\Delta(e)D_2\left(k^{-6}\Delta(e)k^2\Delta(e)k\delta_1(\sigma(e))k\Delta(e)\right)\Delta(e)\delta
_1(\Delta(e)k^2\Delta(e))-i\Delta(e)D_2\left(k^{-6}\Delta(e)k^2\Delta(e)k\delta
_1(\sigma(e))k\Delta(e)\right)\Delta(e)\delta_2(\Delta(e)k^2\Delta(e))+i\Delta(e)
D_2\left(k^{-6}\Delta(e)k^2\Delta(e)k\delta_2(\sigma(e))k\Delta(e)\right)\Delta(e)\delta_1(\Delta(e)k^2\Delta(e))+\Delta(e)
D_2\left(k^{-6}\Delta(e)k^2\Delta(e)k\delta_2(\sigma(e))k\Delta(e)\right)\Delta(e)\delta_2(\Delta(e)k^2\Delta(e))+4
\Delta(e)k^2\Delta(e)D_3\left(k^{-6}\Delta(e)k^2\Delta(e)\delta_1(\Delta(e)k^2\Delta(e))\Delta(e)\right)\Delta(e)\delta
_1(\Delta(e)k^2\Delta(e))+4\Delta(e)k^2\Delta(e)D_3\left(k^{-6}\Delta(e)k^2\Delta(e)\delta
_2(\Delta(e)k^2\Delta(e))\Delta(e)\right)\Delta(e)\delta_2(\Delta(e)k^2\Delta(e))-4\Delta(e)
D_3\left(k^{-8}\Delta(e)k^2\Delta(e)k^2\Delta(e)\delta
_1(\Delta(e)k^2\Delta(e))\Delta(e)\right)\Delta(e)\delta_1(\Delta(e)k^2\Delta(e))-4\Delta(e)
D_3\left(k^{-8}\Delta(e)k^2\Delta(e)k^2\Delta(e)\delta
_2(\Delta(e)k^2\Delta(e))\Delta(e)\right)\Delta(e)\delta_2(\Delta(e)k^2\Delta(e))-2\Delta(e)D_1\left(k^{-4}\Delta(e)k\sigma(e)\delta
_1(k)\Delta(e)\right)\Delta(e)k\sigma(e)\delta_1(k)-\Delta(e)D_1\left(k^{-4}\Delta(e)k\sigma(e)\delta
_1(k)\Delta(e)\right)\Delta(e)k\delta_1(\sigma(e))k-i\Delta(e)D_1\left(k^{-4}\Delta(e)k\sigma(e)\delta
_1(k)\Delta(e)\right)\Delta(e)k\delta_2(\sigma(e))k-2\Delta(e)D_1\left(k^{-4}\Delta(e)k\sigma(e)\delta
_2(k)\Delta(e)\right)\Delta(e)k\sigma(e)\delta_2(k)+i\Delta(e)D_1\left(k^{-4}\Delta(e)k\sigma(e)\delta
_2(k)\Delta(e)\right)\Delta(e)k\delta_1(\sigma(e))k-\Delta(e)D_1\left(k^{-4}\Delta(e)k\sigma(e)\delta
_2(k)\Delta(e)\right)\Delta(e)k\delta_2(\sigma(e))k-\Delta(e)D_1\left(k^{-4}\Delta(e)k\delta
_1(\sigma(e))k\Delta(e)\right)\Delta(e)k\sigma(e)\delta_1(k)+i\Delta(e)D_1\left(k^{-4}\Delta(e)k\delta
_1(\sigma(e))k\Delta(e)\right)\Delta(e)k\sigma(e)\delta_2(k)-i\Delta(e)D_1\left(k^{-4}\Delta(e)k\delta
_2(\sigma(e))k\Delta(e)\right)\Delta(e)k\sigma(e)\delta_1(k)-\Delta(e)D_1\left(k^{-4}\Delta(e)k\delta
_2(\sigma(e))k\Delta(e)\right)\Delta(e)k\sigma(e)\delta_2(k)+2\Delta(e)D_2\left(k^{-6}\Delta(e)k^2\Delta(e)\delta
_1(\Delta(e)k^2\Delta(e))\Delta(e)\right)\Delta(e)k\sigma(e)\delta_1(k)+\Delta(e)D_2\left(k^{-6}\Delta(e)k^2\Delta(e)\delta
_1(\Delta(e)k^2\Delta(e))\Delta(e)\right)\Delta(e)k\delta_1(\sigma(e))k+i\Delta(e)
D_2\left(k^{-6}\Delta(e)k^2\Delta(e)\delta_1(\Delta(e)k^2\Delta(e))\Delta(e)\right)\Delta(e)k\delta_2(\sigma(e))k+2
\Delta(e)D_2\left(k^{-6}\Delta(e)k^2\Delta(e)\delta_2(\Delta(e)k^2\Delta(e))\Delta(e)\right)\Delta(e)k\sigma(e)\delta
_2(k)-i\Delta(e)D_2\left(k^{-6}\Delta(e)k^2\Delta(e)\delta_2(\Delta(e)k^2\Delta(e))\Delta(e)\right)\Delta(e)k\delta
_1(\sigma(e))k+\Delta(e)D_2\left(k^{-6}\Delta(e)k^2\Delta(e)\delta_2(\Delta(e)k^2\Delta(e))\Delta(e)\right)\Delta(e)k\delta
_2(\sigma(e))k+\Delta(e)D_0\left(\frac{1}{k}\frac{1}{k}\Delta(e)k\sigma(e)\delta_1^2(k)\Delta(e)\right)+\Delta(e)D_0\left(\frac{1}{k}\frac{1}{k}\Delta(e)k\sigma(e)\delta
_2^2(k)\Delta(e)\right)+\Delta(e)D_0\left(\frac{1}{k}\frac{1}{k}\Delta(e)k\delta_1(\sigma(e))\delta_1(k)\Delta(e)\right)-i\Delta(e)
D_0\left(\frac{1}{k}\frac{1}{k}\Delta(e)k\delta_1(\sigma(e))\delta_2(k)\Delta(e)\right)+i\Delta(e)D_0\left(\frac{1}{k}\frac{1}{k}\Delta(e)k\delta_2(\sigma(e))\delta
_1(k)\Delta(e)\right)+\Delta(e)D_0\left(\frac{1}{k}\frac{1}{k}\Delta(e)k\delta_2(\sigma(e))\delta_2(k)\Delta(e)\right)-\Delta(e)
D_1\left(k^{-4}\Delta(e)k^2\Delta(e)\delta_1^2(\Delta(e)k^2\Delta(e))\Delta(e)\right)-\Delta(e)
D_1\left(k^{-4}\Delta(e)k^2\Delta(e)\delta_2^2(\Delta(e)k^2\Delta(e))\Delta(e)\right)-2\Delta(e)
D_1\left(k^{-4}\Delta(e)k^2\Delta(e)k\sigma(e)\delta_1^2(k)\Delta(e)\right)-2\Delta(e)
D_1\left(k^{-4}\Delta(e)k^2\Delta(e)k\sigma(e)\delta_2^2(k)\Delta(e)\right)-3\Delta(e)
D_1\left(k^{-4}\Delta(e)k^2\Delta(e)k\delta_1(\sigma(e))\delta_1(k)\Delta(e)\right)+i\Delta(e)
D_1\left(k^{-4}\Delta(e)k^2\Delta(e)k\delta_1(\sigma(e))\delta_2(k)\Delta(e)\right)-\Delta(e)
D_1\left(k^{-4}\Delta(e)k^2\Delta(e)k\delta_1^2(\sigma(e))k\Delta(e)\right)-i\Delta(e)
D_1\left(k^{-4}\Delta(e)k^2\Delta(e)k\delta_2(\sigma(e))\delta_1(k)\Delta(e)\right)-3\Delta(e)
D_1\left(k^{-4}\Delta(e)k^2\Delta(e)k\delta_2(\sigma(e))\delta_2(k)\Delta(e)\right)-\Delta(e)
D_1\left(k^{-4}\Delta(e)k^2\Delta(e)k\delta_2^2(\sigma(e))k\Delta(e)\right)-2\Delta(e)
D_1\left(k^{-4}\Delta(e)k^2\Delta(e)\delta_1(k)\sigma(e)\delta_1(k)\Delta(e)\right)-\Delta(e)
D_1\left(k^{-4}\Delta(e)k^2\Delta(e)\delta_1(k)\delta_1(\sigma(e))k\Delta(e)\right)-i\Delta(e)
D_1\left(k^{-4}\Delta(e)k^2\Delta(e)\delta_1(k)\delta_2(\sigma(e))k\Delta(e)\right)-2\Delta(e)
D_1\left(k^{-4}\Delta(e)k^2\Delta(e)\delta_2(k)\sigma(e)\delta_2(k)\Delta(e)\right)+i\Delta(e)
D_1\left(k^{-4}\Delta(e)k^2\Delta(e)\delta_2(k)\delta_1(\sigma(e))k\Delta(e)\right)-\Delta(e)
D_1\left(k^{-4}\Delta(e)k^2\Delta(e)\delta_2(k)\delta_2(\sigma(e))k\Delta(e)\right)-2\Delta(e)
D_1\left(k^{-4}\Delta(e)k^2\Delta(e)\frac{1}{k}\sigma(e)\delta_1(k)k\sigma(e)\delta_1(k)\Delta(e)\right)-\Delta(e)
D_1\left(k^{-4}\Delta(e)k^2\Delta(e)\frac{1}{k}\sigma(e)\delta_1(k)k\delta_1(\sigma(e))k\Delta(e)\right)-i\Delta(e)
D_1\left(k^{-4}\Delta(e)k^2\Delta(e)\frac{1}{k}\sigma(e)\delta_1(k)k\delta_2(\sigma(e))k\Delta(e)\right)-2\Delta(e)
D_1\left(k^{-4}\Delta(e)k^2\Delta(e)\frac{1}{k}\sigma(e)\delta_2(k)k\sigma(e)\delta_2(k)\Delta(e)\right)+i\Delta(e)
D_1\left(k^{-4}\Delta(e)k^2\Delta(e)\frac{1}{k}\sigma(e)\delta_2(k)k\delta_1(\sigma(e))k\Delta(e)\right)-\Delta(e)
D_1\left(k^{-4}\Delta(e)k^2\Delta(e)\frac{1}{k}\sigma(e)\delta_2(k)k\delta_2(\sigma(e))k\Delta(e)\right)-2\Delta(e)
D_1\left(k^{-4}\Delta(e)k^2\Delta(e)\frac{1}{k}\delta_1(\sigma(e))k^2\sigma(e)\delta_1(k)\Delta(e)\right)-\Delta(e)
D_1\left(k^{-4}\Delta(e)k^2\Delta(e)\frac{1}{k}\delta_1(\sigma(e))k^2\delta_1(\sigma(e))k\Delta(e)\right)-i\Delta(e)
D_1\left(k^{-4}\Delta(e)k^2\Delta(e)\frac{1}{k}\delta_1(\sigma(e))k^2\delta_2(\sigma(e))k\Delta(e)\right)-2\Delta(e)
D_1\left(k^{-4}\Delta(e)k^2\Delta(e)\frac{1}{k}\delta_2(\sigma(e))k^2\sigma(e)\delta_2(k)\Delta(e)\right)+i\Delta(e)
D_1\left(k^{-4}\Delta(e)k^2\Delta(e)\frac{1}{k}\delta_2(\sigma(e))k^2\delta_1(\sigma(e))k\Delta(e)\right)-\Delta(e)
D_1\left(k^{-4}\Delta(e)k^2\Delta(e)\frac{1}{k}\delta_2(\sigma(e))k^2\delta_2(\sigma(e))k\Delta(e)\right)+2\Delta(e)
D_1\left(k^{-4}\Delta(e)k^2\Delta(e)\frac{1}{k}\delta_1(k)\frac{1}{k}\sigma(e)k^2\sigma(e)\delta_1(k)\Delta(e)\right)+\Delta(e)
D_1\left(k^{-4}\Delta(e)k^2\Delta(e)\frac{1}{k}\delta_1(k)\frac{1}{k}\sigma(e)k^2\delta_1(\sigma(e))k\Delta(e)\right)+i\Delta(e)
D_1\left(k^{-4}\Delta(e)k^2\Delta(e)\frac{1}{k}\delta_1(k)\frac{1}{k}\sigma(e)k^2\delta_2(\sigma(e))k\Delta(e)\right)+2\Delta(e)
D_1\left(k^{-4}\Delta(e)k^2\Delta(e)\frac{1}{k}\delta_2(k)\frac{1}{k}\sigma(e)k^2\sigma(e)\delta_2(k)\Delta(e)\right)-i\Delta(e)
D_1\left(k^{-4}\Delta(e)k^2\Delta(e)\frac{1}{k}\delta_2(k)\frac{1}{k}\sigma(e)k^2\delta_1(\sigma(e))k\Delta(e)\right)+\Delta(e)
D_1\left(k^{-4}\Delta(e)k^2\Delta(e)\frac{1}{k}\delta_2(k)\frac{1}{k}\sigma(e)k^2\delta_2(\sigma(e))k\Delta(e)\right)+2\Delta(e)
D_2\left(k^{-6}\Delta(e)k^2\Delta(e)k^2\Delta(e)\delta_1^2(\Delta(e)k^2\Delta(e))\Delta(e)\right)+2\Delta(e)
D_2\left(k^{-6}\Delta(e)k^2\Delta(e)k^2\Delta(e)\delta_2^2(\Delta(e)k^2\Delta(e))\Delta(e)\right) \Big ).
\end{math}
\end{center}

\subsection{Reduction to the flat metric and the Connes-Chern number}

Having calculated explicit local formulas  for $\tau \left ( a_2(L^\pm) \right )$, based on \eqref{indDplus}, we have 
found a local formula for the index of the twisted Dirac operator $D_{e, \sigma}^+$ on $\NT$ given by \eqref{twistedDiracopnct}, 
where the Dirac operator is associated with a conformally flat metric and the twisting is carried out by 
an idempotent playing the role of a general vector bundle. In forthcoming work, we will present a 
simplification of the local formula and will elaborate on the relation between properties of the functions 
of the modular automorphism in the simplified form and stability of the index under perturbations that 
leave the Dirac operator in the same connected component of Fredholm operators. We end this article by  
considering the case of the canonical flat metric on $\NT$, which corresponds to the trivial conformal factor 
$k=1$ whose corresponding modular automorphism $\Delta$ is the identity. Therefore, 
our formula for the index, in this case, reduces to a much simpler form as one can replace 
the functions of the modular automorphism with the values of the functions at 1.  This yields, for $k=1$: 
$$
\textrm{Ind }D_{e,\sigma}^+=2\pi i \, \tau \left (e\delta_1(e)\delta_2(e)-e\delta_2(e)\delta_1(e) \right).
$$
Note that $$c_1(e):=2\pi i \, \tau(e\delta_1(e)\delta_2(e)-e\delta_2(e)\delta_1(e))$$ is the Connes-Chern number
of our projection $e$, and that, 
as proven in \cite{proj}, one can construct $e$ as a self-dual or anti-self dual projection
with  $c_1(e)=n$ for any integer $n$.

\appendix

\section{Proof of Lemma \ref{mvt} and Corollary \ref{mvc}}
\label{appproofsec}

The {\em Hadamard product} of two power series 
$f(z)=\sum_{n=0}^{\infty}f_nz^n$ and $g(z)=\sum_{n=0}^{\infty}g_nz^n$ 
is defined by 
$$
(f\odot g)(z) =\sum_{n=0}^{\infty}f_ng_nz^n.
$$

\begin{prop} \cite[p.~424]{fs}
Suppose $f$ and $g$ are analytic functions on a domain $D\subset\mathbb C$,
and let $\gamma\subset D\cap(zD^{-1})$ be a closed contour. Then
the Hadamard product of $f$ and $g$ obeys the following identity:
$$(f\odot g)(z)=\frac{1}{2\pi i}\int_{\gamma}\!f(w) \, g(z/w) \, dw/w.$$
\end{prop}

\begin{lem} 
Let $N$ be an integer.
For every element $\rho \in \SNT$ and every 
non-negative integer $m$ we have
\begin{align*}
\int_0^{\infty}\!(\sigma^N(ek^2)u+1)^{m+1}\backslash(u^m\sigma^N(e)\rho)/(\sigma^N(ek^2)u+1)\,\mathrm du & \\
=\sigma^N(e)  D_m(k^{-(2m+2)}\sigma^N(e)\rho \sigma^N(e)), &
\end{align*}
where the modified logarithm $ D_m=\mathcal L_m(\Delta)$ of the modular automorphism 
$\Delta(a)=k^{-2}ak^2$, $a \in \CNT$, with the square root $\sigma(a)=k^{-1}ak$, is defined via the function
\begin{align*}
\mathcal L_m(u)&=\int_0^{\infty}\!\frac{x^m}{(x+1)^{m+1}}\frac{1}{xu+1}\,\mathrm dx \\
&=(-1)^m(u-1)^{-(m+1)}\left(\log u-\sum_{j=1}^m(-1)^{j+1}\frac{(u-1)^j}{j}\right). 
\end{align*}
\end{lem}
\begin{proof} First we look at the $N=0$ case.
We express the integrand as two Hadamard products \cite[p.~424]{fs} of generating functions multiplied together, and then we make
the change of variables $u=\exp(s)$. We will substitute an  improper integral using 
$$\frac{\exp[(h+s)/2]}{1+\exp(h+s-a-i\phi)}
=\exp[i(\phi-ia)/2]\int_{-\infty}^{\infty}\!\frac{\exp[it(h+s-a-i\phi)]}{\exp(\pi t)+\exp(-\pi t)}\,\mathrm dt$$
for the closed form of the Fourier transform of the hyperbolic secant, and will express an inner
integral as the Fourier transform of another function. Let $a$ be a positive real number. We get
\begin{eqnarray*}
&&\int_0^{\infty}\!(ek^2u+1)^{m+1}\backslash(u^me\rho)/(ek^2u+1)\,\mathrm du \\
&&=\int_0^{\infty}\!(ek^2u+1)^{m+1}\backslash(k^{2m+2}u^m(k^{-(2m+2)}e\rho))/(ek^2u+1)\,\mathrm du \\
&&=\int_0^{\infty}\!(k^2u)^{m+1/2}\left[\frac{1}{(1+k^2u)^{m+1}}
\odot\Delta^{m+1/2}\left(\left(\frac{1}{1-u\Delta R_e}(1)\right)\right)\right] \\
&&\Delta^{-1/2}(k^{-(2m+2)}e\rho)(k^2u)^{1/2} 
\end{eqnarray*}
\begin{eqnarray*}
&& \left[\frac{1}{1+k^2u}\odot\frac{1}{1-u\Delta R_e}(1)\right]\frac{\mathrm du}{u} \\
&&=\int_{-\infty}^{\infty}\!\left(\left[\frac{1}{2\pi}\int_0^{2\pi}\!\frac{\exp[(m+1/2)(s+h)]}{[1+\exp(h+s-a-i\theta)]^{m+1}} \right.\right. \\
&&\left.\Delta^{m+1/2}\left(\left(\frac{1}{1-\Delta R_e\exp(i\theta+a)}(1)\right)\right)\,\mathrm d\theta\right]\cdot\Delta^{-1/2}(k^{-(2m+2)}e\rho) \\
&&\left.\left[\frac{1}{2\pi}\int_0^{2\pi}\!\frac{\exp[(s+h)/2]}{1+\exp(h+s-a-i\phi)}\frac{1}{1-\Delta R_e\exp(i\phi+a)}(1)
\,\mathrm d\phi\right]\right)\,\mathrm ds 
\end{eqnarray*}
\begin{eqnarray*}
&&=\frac{1}{(2\pi)^2}\int_{-\infty}^{\infty}\!\frac{\exp(ith)}{\exp(\pi t)+\exp(-\pi t)}\left(\int_{-\infty}^{\infty}\!
\left[\int_0^{2\pi}\!\frac{\exp[(m+1/2)(s+h)]}{[1+\exp(h+s-a-i\theta)]^{m+1}} \right.\right. \\
&&\left.\Delta^{m+\frac12+it}\left(\left(\frac{1}{1-\Delta R_e\exp(i\theta+a)}(1)\right)\right)\,\mathrm d\theta\right]
\Delta^{-\frac12+it}(k^{-(2m+2)}e\rho) \\
&&\left.\left[\int_0^{2\pi}\!\frac{\exp[\frac{i\phi+a}{2}+its+t(\phi-ia)]}{1-\Delta R_e\exp(i\phi+a)}(1)\,\mathrm d\phi\right]
\mathrm ds\right)\mathrm dt 
\end{eqnarray*}
\begin{eqnarray*}
&&=\frac{1}{(2\pi)^2}\int_{-\infty}^{\infty}\!\frac{\exp(ith)}{\exp(\pi t)+\exp(-\pi t)}
\left[\int_0^{2\pi}\!\left(\int_{-\infty}^{\infty}\!\frac{\exp[(m+1/2)(s+h)]\exp(its)}{[1+\exp(h+s-a-i\theta)]^{m+1}}\,\mathrm ds\right)\right. \\
&&\left.\Delta^{m+\frac12+it}\left(\left(\frac{1}{1-\Delta R_e\exp(i\theta+a)}(1)\right)\right)\,\mathrm d\theta\right]\Delta^{-\frac12+it}(k^{-(2m+2)}e\rho) \\
&& \left[\int_0^{2\pi}\!\frac{\exp[(i\phi+a)/2+t(\phi-ia)]}{1-\Delta R_e\exp(i\phi+a)}(1)\,\mathrm d\phi\right]
\mathrm dt. 
\end{eqnarray*}
For the second equality, agreement of the integrands within the radii of convergence of the
geometric series implies agreement of the integrands throughout their analytic continuations: 
\begin{eqnarray*}
&&\left[\sum_{n=0}^{\infty}(-ek^2u)^n\right]^{m+1}k^{2m+2}u^m(k^{-(2m+2)}e\rho)
\left[\sum_{n=0}^{\infty}(-ek^2u)^n\right]  \\
&&=\left[\sum_{n_1=0}^{\infty}\cdots\sum_{n_{m+1}=0}^{\infty}\prod_{q=1}^{m+1}(-k^2u)^{n_q}
\left(\prod_{p=1}^{n_q}\Delta^pe\right)^*\right]  \\
&&\cdot k^{2m+2}u^m(k^{-(2m+2)}e\rho)\left[\sum_{n=0}^{\infty}(-k^2u)^n\left(\prod_{p=1}^n\Delta^pe\right)^*\right]  
\end{eqnarray*}
\begin{eqnarray*}
&&=\left[\sum_{n_1=0}^{\infty}\cdots\sum_{n_{m+1}=0}^{\infty}(-k^2u)^{\sum_{q=1}^{m+1}n_q}
\left(\prod_{p=1}^{\sum_{q=1}^{m+1}n_q}\Delta^pe\right)^*\right]  \\
&&\cdot k^{2m+2}u^m(k^{-(2m+2)}e\rho)\left[\sum_{n=0}^{\infty}(-k^2u)^n\left(\prod_{p=1}^n\Delta^pe\right)^*\right]. 
\end{eqnarray*}
Since
\begin{align*}
&\int_{-\infty}^{\infty}\!\frac{\exp[(m+\frac12)(s+h)]\exp(its)}{[1+\exp(h+s-a-i\theta)]^{m+1}}\,\mathrm ds \\
&=\exp[i(m+\frac12)(\theta-ai)]\int_{-\infty}^{\infty}\!\frac{\exp[(m+\frac12)(h+s-a-i\theta)]\exp(its)}{[1+\exp(h+s-a-i\theta)]^{m+1}}\,\mathrm ds \\
&=\exp[i(m+\frac12)(\theta-ai)]\exp[-it(h-a-i\theta)]\hat{f}_m(t),
\end{align*}
where $\hat{f}_m(t)$ is the Fourier transform of the function
$$f_m(s)=\frac{\exp[(m+1/2)s]}{[\exp(s)+1]^{m+1}},$$
we have:
\begin{eqnarray*}
&&\int_0^{\infty}\!(ek^2u+1)^{m+1}\backslash(u^me\rho)/(ek^2u+1)\,\mathrm du \\
&&= \frac{1}{(2\pi)^2}\int_{-\infty}^{\infty}\frac{\hat{f}_m(t)}{\exp(\pi t)+\exp(-\pi t)} \\
&&\Delta^{m+\frac12+it}
\left[\left(\left(\int_0^{2\pi}\!\frac{\exp[i(m+1/2+it)(\theta-ai)]}{1-\Delta R_e\exp(i\theta+a)}\,\mathrm d\theta\right)(1)\right)\right] 
\end{eqnarray*}
\begin{eqnarray*}
&&\Delta^{-\frac12+it}(k^{-(2m+2)}e\rho)
\left[\left(\int_0^{2\pi}\frac{\exp[i(1/2-it)(\phi-ia)]}{1-\Delta R_e\exp(i\phi+a)}\,\mathrm d\phi\right)(1)\right]\,\mathrm dt \\
&&= \frac{1}{(2\pi)^2}\int_{-\infty}^{\infty}\frac{\hat{f}_m(t)}{\exp(\pi t)+\exp(-\pi t)} \\
&&\left[\left(\left(\int_0^{2\pi}\!\frac{\exp[i(m+1/2+it)(\nabla+\theta-ai)]}{1-\Delta R_e\exp(i\theta+a)}\,\mathrm d\theta\right)(1)\right)\right] \\
&&\Delta^{-\frac12+it}\left(k^{-(2m+2)}e\rho
\left[\left(\int_0^{2\pi}\frac{\exp[i(1/2-it)(\nabla+\phi-ia)]}{1-\Delta R_e\exp(i\phi+a)}\,\mathrm d\phi\right)(1)\right]\right)\,\mathrm dt.
\end{eqnarray*}
For $e=1$, the inner integrals evaluate to $2\pi$, so what we get agrees with the previous result.
For $0\ne w\in\mathbb C\backslash(-\infty,0]$
and $a>-\ln|w|$, we integrate over a Hankel contour $H_{\epsilon, R}$ with branch cut along the negative $x$-axis and get
\begin{eqnarray*}
&&\int_0^{2\pi}\!\frac{\exp[i(m+1/2+it)(\nabla+\theta-ai)]}{1-w\exp(i\theta+a)}\,\mathrm d\theta \\
&&= \int_{C_{\exp{a}}}\frac{(\Delta z)^{m+1/2+it}}{1-zw}\frac{1}{i}\frac{\mathrm dz}{z} 
\end{eqnarray*}
\begin{eqnarray*}
&&= \frac{1}{i}\lim_{R\rightarrow\infty}\lim_{\epsilon\rightarrow 0}\int_{H_{\epsilon,R}}
\frac{(\exp[(m-1/2+it)(\log z+\nabla)])}{1-zw}\,\mathrm dz\,\Delta \\
&&= 2\pi(\exp[(m-1/2+it)\log(w^{-1}\Delta)])\Delta \\
&&= 2\pi((\Delta^{-1}w)^{\frac{1}{2}-m-it})\Delta, 
\end{eqnarray*}
and
\begin{align*}
&\int_0^{2\pi}\!\frac{\exp[i(1/2-it)(\nabla+\phi-ai)]}{1-w\exp(i\phi+a)}\,\mathrm d\phi \\
&= \int_{C_{\exp{a}}}\frac{(\Delta z)^{1/2-it}}{1-zw}\frac{1}{i}\frac{\mathrm dz}{z} \\
&= \frac{1}{i}\lim_{R\rightarrow\infty}\lim_{\epsilon\rightarrow 0}\int_{H_{\epsilon,R}}
\frac{(\exp[(-1/2-it)(\log z+\nabla)])}{1-zw}\,\mathrm dz\,\Delta \\
&= 2\pi(\exp[(-1/2-it)\log(w^{-1}\Delta)])\Delta \\
&= 2\pi(\Delta^{-1}w)^{\frac{1}{2}+it}\Delta.
\end{align*}
Since $R_e$ is an idempotent operator, we can reduce
to the previous result as follows:
\begin{eqnarray*}
&&\int_0^{\infty}\!(ek^2u+1)^{m+1}\backslash(u^me\rho)/(ek^2u+1)\,\mathrm du \\
&&=\int_{-\infty}^{\infty}\frac{\hat{f}_m(t)(\Delta^{-1}\Delta R_e)^{\frac12-m-it}\Delta(1)
\Delta^{-\frac12+it}(k^{-(2m+2)}e\rho(\Delta^{-1}\Delta R_e)^{\frac12+it}\Delta(1))\,\mathrm dt}{\exp(\pi t)+\exp(-\pi t)} \\
&&=\int_{-\infty}^{\infty}\frac{\hat{f}_m(t)}{\exp(\pi t)+\exp(-\pi t)}
R_e(1)\Delta^{-\frac12+it}(k^{-(2m+2)}e\rho R_e(1))\,\mathrm dt 
\end{eqnarray*}
\begin{eqnarray*}
&&=\int_{-\infty}^{\infty}\frac{\hat{f}_m(t)}{\exp(\pi t)+\exp(-\pi t)}e
\Delta^{-\frac12+it}(k^{-(2m+2)}e\rho e)\,\mathrm dt \\
&&=e\int_{-\infty}^{\infty}\frac{\hat{f}_m(t)}{\exp(\pi t)+\exp(-\pi t)}
\Delta^{-\frac12+it}(k^{-(2m+2)}e\rho e)\,\mathrm dt \\
&&=e D_m(k^{-(2m+2)}e\rho e).
\end{eqnarray*}
Now that we have proven the lemma for $N=0$, we may replace $\rho$ with $\sigma^{-N}(\rho)$ and apply $\sigma^N$ to both sides, which
proves the lemma for any integer $N$.
\end{proof}

\begin{cor} Let $N$ be an integer. Then
\begin{align*}
\int_0^{\infty}\!(\sigma^N(ek^2e)u+1)^{m+1}\backslash(u^m\sigma^N(e)\rho)/(\sigma^N(ek^2e)u+1)\,\mathrm du & \\
=\sigma^N(e) D_m(k^{-(2m+2)}\sigma^N(e)\rho \sigma^N(e))\sigma^N(e) &
\end{align*}
\end{cor}
\begin{proof}
The $N=0$ case of lemma we just proved implies that, for any $\rho\in A_{\theta}^{\infty}$
and every non-negative integer $m$, we have
$$\int_0^{\infty}\!(ek^2u+1)^{m+1}\backslash(u^me\rho)/(ek^2u+1)\,\mathrm du
=e D_m(k^{-(2m+2)}e\rho e)$$
and
$$\int_0^{\infty}\!(ek^2u+1)^{m+1}\backslash(u^me\rho e)/(ek^2u+1)\,\mathrm du
=e  D_m(k^{-(2m+2)}e\rho e).$$
Since $1/(ek^2u+1)-e/(ek^2u+1)=1-e$, subtracting the two equations immediately above gives us
$$\int_0^{\infty}\!(ek^2u+1)^{m+1}\backslash(u^me\rho(1-e))\,\mathrm du=0.$$
Also,
$$\int_0^{\infty}\!(ek^2u+1)^{-(m+1)}u^me\rho(ek^2u+1)^{-1}e\,\mathrm du
=e D_m(k^{-(2m+2)}e\rho e)e.$$
Since $1/(ek^2eu+1)=1-e+(ek^2u+1)\backslash e$, adding the two equations immediately above
gives us
$$\int_0^{\infty}\!(ek^2u+1)^{m+1}\backslash(u^me\rho)/(ek^2eu+1)\,\mathrm du
=e  D_m(k^{-(2m+2)}e\rho e)e.$$
Since $(ek^2u+1)^{m+1}\backslash e=(ek^2eu+1)^{m+1}\backslash e$, we get
$$\int_0^{\infty}\!(ek^2eu+1)^{m+1}\backslash(u^me\rho)/(ek^2eu+1)\,\mathrm du
=e D_m(k^{-(2m+2)}e\rho e)e.$$
After replacing $\rho$ with $\sigma^{-N}(\rho)$ and applying $\sigma^N$ to both sides, we are done.
\end{proof}

We also state a corollary of our rearrangement lemma regarding orthogonal projections. 
Suppose $e_1,e_2\in A_{\theta}^{\infty}$ are orthogonal self-adjoint idempotents,
i.e.~$e_1^2=e_1=e_1^*$, $e_2^2=e_2=e_2^*$, and $e_1e_2=e_2e_1=0$. Then $(e_1+e_2)^2=e_1^2+e_2^2=e_1+e_2
=(e_1+e_2)^*$, so $e_1+e_2$ is a self-adjoint idempotent and the following corollary follows from our rearrangement lemma:  
\begin{cor}
For every element $\rho\in A_{\theta}^{\infty}$ and every non-negative integer $m$ we have
\begin{align*}
\int_0^{\infty}\!(\sigma^N(e_1+e_2)k^2u+1)^{m+1}\backslash(u^m\sigma^N(e_1+e_2)\rho)/(\sigma^N(e_1+e_2)k^2u+1)\,\mathrm du & \\
=\sigma^N(e_1+e_2)  D_m(k^{-(2m+2)}\sigma^N(e_1+e_2)\rho \sigma^N(e_1+e_2)). &
\end{align*}
\end{cor}

\section*{Acknowledgments}
F. F. acknowledges the support from the Marie
Curie/SER Cymru II Cofund Research Fellowship 663830-SU-008. 
J. T.  acknowledges the support from the 2017-2018 Perimeter 
Institute Visiting Graduate Fellowship. The authors are indebted to 
Matilde Marcolli for discussions, and her support and encouragements.

\end{document}